\documentclass[11pt]{amsart}

\usepackage[dvipsnames]{xcolor}
\usepackage[english]{babel}
\usepackage[T1]{fontenc}
\usepackage[utf8]{inputenc}
\usepackage{amsmath}
\usepackage{amssymb}
\usepackage{amsthm}
\usepackage{color}
\usepackage[a4paper]{geometry}
\usepackage{tikz}
\usepackage{mathdots}
\usepackage[all]{xy}
\usepackage{hyperref}
\usepackage{accents}
\usepackage{lmodern}
\usepackage{mathrsfs}
\usepackage{appendix}
\usepackage{mathtools,xparse}

\theoremstyle{definition}
\newtheorem{definition}{Definition}[section]

\newtheorem{rmk}[definition]{Remark}

\theoremstyle{plain}
\newtheorem{theorem}[definition]{Theorem}
\newtheorem{prop}[definition]{Proposition}
\newtheorem{cor}[definition]{Corollary}
\newtheorem{lem}[definition]{Lemma}

\newcommand{\mb}{\mathbb}
\newcommand{\mc}{\mathcal}

\newcommand{\bs}{\boldsymbol}

\newcommand{\SL}{\textup{SL}}
\newcommand{\R}{\mathbb{R}}
\newcommand{\Z}{\mathbb{Z}}
\newcommand{\N}{\mathbb{N}}
\newcommand{\dd}{\textup{d}}

\newcommand{\ignore}[1]

\newif\ifdraft\drafttrue

\draftfalse

\newcommand\eq[2]{{\ifdraft{\ \tt [#1]}\else\ignorespaces\fi}\begin{equation}\label{eq:#1}{#2}\end{equation}}
\newcommand {\equ}[1]     {\eqref{eq:#1}}

\numberwithin{equation}{section}

\usepackage{scalerel,stackengine}
\stackMath
\newcommand\reallywidehat[1]{%
\savestack{\tmpbox}{\stretchto{%
  \scaleto{%
    \scalerel*[\widthof{\ensuremath{#1}}]{\kern-.6pt\bigwedge\kern-.6pt}%
    {\rule[-\textheight/2]{1ex}{\textheight}}
  }{\textheight}%
}{0.5ex}}%
\stackon[1pt]{#1}{\tmpbox}%
}
\begin{document}

\thanks{RF was supported by SNSF grant 200021--182089 and LMS Early Career Fellowship ECF-1920-15. DK was supported by NSF grants DMS-1900560 and DMS-2155111. This material is based upon work supported by a grant from the Institute for Advanced Study School
of Mathematics.}

\title{On Multiplicatively Badly Approximable Vectors }

\author[Reynold Fregoli]{Reynold Fregoli}
\address{{Department of Mathematics, University of Z\"urich, Switzerland}} 
\email{{reynold.fregoli@math.uzh.ch}}

\author[Dmitry Kleinbock]{Dmitry Kleinbock}
\address{Department of Mathematics, Brandeis University, 
Waltham MA, USA} 
\email{kleinboc@brandeis.edu}


\begin{abstract}
{Let $\langle x\rangle$ denote the distance from $x\in\mb{R}$ to the set of integers $\mb{Z}$.}
The Littlewood Conjecture states that for all pairs 
$(\alpha,\beta)\in\mb{R}^{2}$ the product $q\langle q\alpha\rangle\langle q\beta\rangle$ attains values arbitrarily close to $0$ as $q\in\mb{N}$ tends to infinity. Badziahin showed that if a factor $\log q\cdot \log\log q$ is added to the product, the same statement becomes false. In this paper, we generalise Badziahin's result to vectors $\bs{\alpha}\in\mb{R}^{d}$, replacing the function $\log q\cdot \log\log q$ by $(\log q)^{d-1}\cdot\log\log q$ for any $d\geq 2$, and thereby obtaining a new proof in the case $d=2$. Our approach is based on a new version of the well-known Dani Correspondence {between Diophantine approximation and} dynamics on the space of lattices, especially adapted to the study of products of rational approximations. We believe that this correspondence is of independent interest.
\end{abstract}

\maketitle

\section{Introduction}

\subsection{{The} Main Result}

The Littlewood Conjecture states that for all pairs of real numbers $(\alpha,\beta)\in\mb{R}^{2}$
\begin{equation}
\label{eq:Lit}
\liminf_{\substack{q\to \infty \\ q\in\mb Z}}q\langle q\alpha\rangle\langle q\beta\rangle=0,
\end{equation}
where the expression $\langle\cdot\rangle$ denotes the distance to the nearest integer. {This conjecture is to date} widely open. It is 
{well-known}  that (\ref{eq:Lit}) holds for Lebesgue-almost-every pair $(\alpha,\beta)\in\mb R^2$. {Indeed, if $(\alpha_0,\beta_0)$ is a counterexample to (\ref{eq:Lit}), then both $\alpha_0$ and $\beta_0$ are forced to be badly approximable -- a zero measure condition. }
Some stronger partial results are also known to hold \cite{CasSwy, PoVe}. {Perhaps the most striking of these} was 
obtained in \cite{EKL}, where it was established that the set of pairs $(\alpha,\beta)$ contradicting (\ref{eq:Lit}) has Hausdorff dimension 
zero. 

A more general version of the Littlewood Conjecture asserts that for all $d\geq 2$ and all vectors ${\bs{\alpha} = (\alpha_1,\dots,\alpha_d)}\in\mb{R}^{d}$ one has that
\begin{equation}
\label{eq:Litgen}
\liminf_{q\to \infty}q\langle q\alpha_{1}\rangle \dotsm\langle q\alpha_{d}\rangle =0.
\end{equation}
When $d\geq 3$, this statement is 
{implied by} (\ref{eq:Lit}). Despite this, (\ref{eq:Litgen}) remains in question even for $d\geq 3$. The vectors $\bs{\alpha}\in\mb{R}^{d}$ which do not satisfy (\ref{eq:Litgen}) are known in the literature as \textit{multiplicatively badly approximable} \cite{Bug, LV}.

In this paper we will be investigating a partial converse of (\ref{eq:Litgen}), where the factor $q$ is replaced with an increasing function $f(q)$. More specifically, we will be interested in determining the minimal growth rate of the function $f$ for which there exist vectors $\bs{\alpha}\in\mb{R}^{d}$ satisfying 
\begin{equation}
\label{eq:Litconv}
\liminf_{q\to \infty}f(q)\langle q\alpha_{1}\rangle \dotsm\langle q\alpha_{d}\rangle >0.
\end{equation}

This problem has a rich history. The earliest known result in this regard was established by Gallagher \cite{Gal}, who showed that if the series $\sum_{q\in\mb{N}}f(q)^{-1}(\log q)^{d-1}$ diverges, Inequality (\ref{eq:Litconv}) fails for Lebesgue-almost-every vector $\bs{\alpha}\in\mb{R}^{d}$. Conversely, the Borel-Cantelli Lemma easily implies that if the same series converges, (\ref{eq:Litconv}) holds for almost every $\bs{\alpha}\in\mb{R}^{d}$. In view of this, studying (\ref{eq:Litconv}) becomes particularly interesting under the assumption that the series $\sum_{q\in\mb{N}}f(q)^{-1}(\log q)^{d-1}$ diverges. If this occurs, the set of vectors $\bs{\alpha}\in\mb{R}^{d}$ satisfying (\ref{eq:Litconv}) is Lebesgue-null, and to determine whether such set is non-empty, it is often necessary to construct some Cantor-type set. This approach has been employed in the literature to estimate the Hausdorff dimension of a variety of "limsup sets" defined by Diophantine properties (see, e.g., \cite{DRV}).

The first progress towards establishing (\ref{eq:Litconv}) in the divergence case is due to Moshchevitin and Bugeaud \cite{BM}, who proved that the set of pairs $(\alpha,\beta)\in[0,1]^{2}$ for which
\begin{equation}
\label{eq:MosBug}
\liminf_{q\to \infty}q(\log q)^{2}\cdot \langle q\alpha\rangle\langle q\beta\rangle>0
\end{equation}
has full Hausdorff dimension (note that $\sum_{q}(q\log q)^{-1}$ is a divergent series). A few years later, Badziahin \cite{Bad} could significantly improve upon (\ref{eq:MosBug}), showing that the set of 
$(\alpha,\beta)\in\mb{R}^{2}$ satisfying the stronger condition
\begin{equation}
\label{eq:Badz}
\liminf_{q\to \infty}q\log q\log\log q\cdot \langle q\alpha\rangle\langle q\beta\rangle>0
\end{equation}
also has full Hausdorff dimension. To the authors' knowledge, no further progress has been made in decreasing the growth rate of the function $f$ or extending the result to higher-dimension, up until the present day. It is worth noting that the minimal growth rate for $f$, when $n=2$, is conjectured to be of order $\log x$ (see \cite[Conjecture L2]{BV}). 


Our main result in the present paper is a full generalisation of (\ref{eq:Badz}) to the case $d\geq 3$.

\begin{theorem}
\label{thm:thm1}
Let $d\geq 1$. Then for any box $B\subset[0,1]^{d}$ the set of vectors $\bs{\alpha}\in B$ such that
$$\liminf_{q\to \infty}q (\log q)^{d-1}\log\log q\cdot\langle q\alpha_{1}\rangle \dotsm\langle q\alpha_{d}\rangle >0$$
has full Hausdorff dimension.
\end{theorem}

Theorem \ref{thm:thm1} relies on a new multiplicative version of the so-called Dani correspondence: in essence, a way of recasting the problem in terms of the dynamics of the diagonal action on the space of unimodular lattices. The advantage of this approach consists in the fact that it naturally extends from dimension $2$ to general dimension $d$, unlike the methods of Badziahin, which appear to apply only in dimension $2$. We are also able to extend our construction to the case of linear forms in arbitrary dimension, as we proceed to explain in the following sub-section.

\subsection{The Dual Case} 

Form now on, let $|x|_+$ denote the quantity $\max (|x|,1)$ for $x\in\R$, and let $\Pi_+(\bs{x})$ denote the product $\prod_{x_i\ne 0}|x_i|$ for any vector $\bs{x} = (x_1,\dots,x_d)\in\R^d$. Note that, when $\bs{x}\in\Z^d$, we have that 
$\Pi_+(\bs{x}) = \prod_{i}|x_i|_+$. The Littlewood conjecture admits a well known dual version, where the pair $(\alpha,\beta)$ plays the role of a linear form. Namely, for all pairs $(\alpha,\beta)\in\mb{R}^{2}$ it is conjectured that
\begin{equation}
\label{eq:dualLit}
\liminf_{|\bs{q}|\to \infty}\Pi_+(\bs{q})
\langle\alpha q_{1}+\beta q_{2}\rangle =0,
\end{equation}
where 
$\bs{q}=(q_1,q_2)\in\mb{Z}^{2}$. It is proved in \cite{CasSwy} that (\ref{eq:dualLit}) is in fact equivalent to (\ref{eq:Lit}). Equation (\ref{eq:Lit}) is often referred to as the "simultaneous case" of the conjecture, as opposed to the "dual case" (\ref{eq:dualLit}).

In \cite{Bad}, Badziahin proved that the set of pairs $(\alpha,\beta)\in\mb{R}^{2}$ simultaneously satisfying (\ref{eq:Badz}) and the "dual" inequality
\begin{equation}
\label{eq:BadzDual}
\liminf_{|\bs{q}|\to \infty}\Pi_+(\bs{q})\log^{+}\left(q_{1}q_{2}\right) \log^{+}\log^{+}\left(q_{1}q_{2}\right) \langle\alpha q_{1}+\beta q_{2}\rangle >0
\end{equation}
has full Hausdorff dimension. Here and hereafter, we use the symbol $\log^{+}(x)$ to denote the function $\log \max\{x,e\}$ for any $x\geq 0$. The most general form of our result extends also (\ref{eq:BadzDual}) to higher dimension.

Let us introduce some notation. Let $d\geq 1$, 
and let $h:\mb{N}\to {(0,\infty)}$ be any function. Consider the sets\footnote{The name Mad was proposed by Badziahin and Velani in \cite{BV} and stands for multiplicatively badly approximable.}
$${\textup{Mad}(d
,h):=\left\{\bs{\alpha}\in \R^{d}:{\inf_{q\ne0}|q| h\left(|q|\right)}\langle q\alpha_{1}\rangle \dotsm\langle q\alpha_{d}\rangle >0\right\}}$$
and
$${\textup{Mad}^{*}(d,
h):=\left\{\bs{\alpha}\in \R^{d}:{\inf_{\bs{q}\ne\bs{0}}}\Pi_+(\bs{q})h\left(\Pi_+(\bs{q})\right)\langle\bs{q}\cdot\bs{\alpha}\rangle >0\right\}.}$$

We show that the intersection of these two sets has full Hausdorff dimension for $d\geq 2$ and $$h(x)=h_{d}(x):=(\log^{+}x)^{d-1}\log^{+}\log^{+}x.$$ Additionally, we prove a similar result for the set of vectors $\bs{\alpha}\in\mb{R}^{d}$ such that each $l$-dimensional sub-vector of $\bs{\alpha}$ lies in the set $\textup{Mad}(l
,h_l)\cap\textup{Mad}^{*}(l,h_l)$ for $l=1,\dotsc,d$. {For a nonempty subset $S$ of $\{1,\dotsc,d\}$ let us denote by} $\pi_{S}:\mb{R}^{d}\to \mb{R}^{\#S}$   the projection onto the coordinates with indices lying in the set $S$ (with the convention that $\pi_{\{1,\dotsc,d\}}=\textup{id}$). The precise statement of our most general result reads as follows.

\begin{theorem}
\label{prop:mainres}
Let $d\geq 2$. Then {the intersection of the set}
\begin{equation}
\label{eq:mainres}
{\bigcap_{l=1}^{d}\bigcap_{\#S=l}\pi_{S}^{-1}\left(\textup{Mad}(l,
h_{l}
)\cap\textup{Mad}^{*}(l,
h_{l}
)\right)
}
\end{equation} {with  any box $B\subset
\R^{d}$} 
has full Hausdorff dimension.
\end{theorem}

Theorem \ref{prop:mainres} clearly implies Theorem \ref{thm:thm1}. 

\begin{rmk}
\label{rmk:ref}
It should be noted that, even in dimension $d=2$, this result is new. In fact, it is shown in \cite[Theorem 2]{Bad} that for any $\alpha\in\textup{Bad}$ (i.e., for any badly approximable number) the set
$$\textup{Mad}(
h_{2}
)\cap\textup{Mad}^{*}(
h_{2})\cap \left(\alpha\times [0,1)\right)$$
has full Hausdorff dimension. In Theorem \ref{prop:mainres}, on the other hand, it is enough to assume that $\alpha\in \textup{Mad}(1,h_1)\supsetneq\textup{Bad}$. To explain the reason for this strengthening, we recall that Badziahin's approach (see \cite{Bad}) consists in removing intervals of the form
$$\left\{x\in[0,1):\left|x-\frac{p}{q}\right|\leq\frac{1}{\log q\log\log q \cdot q\langle q\alpha\rangle}\right\}$$
for different integers $p,q$, with $q\neq 0$. The width of each interval depends on the magnitude of the product $q\langle q\alpha\rangle$ and, thus, this quantity needs to be tightly controlled (in particular, bounded below). In this paper, conversely, sets are removed according to a new correspondence with dynamics (see Section \ref{sec:MultDani}), where the quantity $q\langle q\alpha\rangle$ is replaced by an expression of the form $e^{t+t_1}$, which requires no lower bound. Here, $t$ and $t_1$ denote certain times in a diagonal flow of rank $2$. The important role played by the sum $t+t_1$ will be apparent in Section \ref{sec:countingonav}.
\end{rmk}

\subsection{Further Applications}\label{appl}

{Fix $m,n\in\mb{N}$ and ${Y = (y_{ij})}\in\mb{R}^{m\times n}$. For $\bs{Q} = (Q_{1},\dots, \\ Q_{n})$ $\in[1,\infty)^{n}$} define
$$X(\bs{Q}):=\prod_{i=1}^{n}[-Q_{i},Q_{i}]$$
and put
\begin{equation}
S_{{{Y}}}(\bs{Q}):=\sum_{\substack{\bs{q}\in X(\bs{Q}) \\ \bs{q}\neq\bs{0}}}\frac{1}{\langle{Y}_{1}
\bs{q}\rangle \dotsm\langle{Y}_{m}
\bs{q}\rangle }\nonumber
\end{equation}
{(here and hereafter ${Y}_{i}$ denotes the $i$-th row of the matrix ${{Y}}$)}. 
To ensure that 
$S_{{{Y}}}(\bs Q)$ is well-defined, let us assume that for each $i=1,\dotsc, m$ the numbers {$1,y_{i1},\dotsc, y_{in}$} are linearly independent over $\mb Z$.

Functions such as $S_{{{Y}}}(\bs Q)$ are known in the literature as sums of reciprocals of fractional parts and are widely studied both in Diophantine approximation and in the theory of uniform distribution modulo $1$ (see \cite{BHV} and \cite{Bec94}). In general, one is interested in establishing bounds for 
$S_{{{Y}}}(\bs{Q})$ when the matrix ${{Y}}$ is "typical" and in determining the deviation from the expected growth rate for "exceptional" choices of $Y$. In \cite{LV}, L{\^e} and Vaaler proved a general lower bound for the growth rate of 
$S_{{{Y}}}$, showing that for \emph{all} matrices ${{Y}}\in\mb R^{m\times n}$ it holds that
\begin{equation}
\label{eq:LeVa}
S_{{{Y}}}(\bs{Q})\geq c\cdot (Q_{1}\dotsm Q_{n})^{n}\log(Q_{1}\dotsm Q_{n})^{m},
\end{equation}
where the constant $c>0$ only depends on $m$ and $n$ \cite[Corollary 1.2]{LV}. L\^e and Vaaler further proved that the converse inequality holds when the matrix ${{Y}}$ is multiplicatively badly approximable (i.e., contradicts a more general version of (\ref{eq:Litgen})). Since the existence of such matrices is not known, they posed the question of whether the bound in (\ref{eq:LeVa}) is sharp. 

In a series of works by Widmer and the first author \cite{Wid, Fre1, Fre3}, it was shown that multiplicative bad approximability is an unnecessarily restrictive condition to prove the sharpness of (\ref{eq:LeVa}). The most general result in this direction is \cite[Theorem 1.3]{Fre3}, which we recall here for the convenience of the reader: let $\phi$ be some {positive} non-increasing function. Then for any matrix ${{Y}}\in\mb R^{m\times n}$ satisfying the inequality
\begin{equation}
\label{eq:multbad}
\Pi_+(\bs{q})\langle{Y}_1
 \bs{q}\rangle \dotsm\langle{Y}_m
  \bs{q}\rangle \geq\phi\left(\Pi_+(\bs{q})\right)
\end{equation}
for all $\bs{q}\in\mb{Z}^{n}\setminus\{\bs{0}\}$, it holds that
\begin{equation}
\label{eq:fregobound}
S_{{{Y}}}(\bs{Q})\leq c'\cdot \left( \bar Q^{n}\left(\log\frac{\bar Q}{\phi(\bar Q)}\right)^{m}+\frac{\bar Q^{n}}{\phi(\bar Q)}\left(\log\frac{\bar Q}{\phi(\bar Q)}\right)^{m-1}\right) 
\end{equation}
for all $\bs Q\in(2,+\infty]^n$.
Here, $c'>0$ is a constant only depending on $m$ and $n$, and $\bar Q$ is defined as $(Q_1\dotsm Q_n)^{1/n}$.

In view of (\ref{eq:fregobound}), Theorem \ref{prop:mainres} admits the following straightforward corollary.

\begin{cor}
\label{cor:cor}
There exists a full Hausdorff dimension set of vectors $\bs{\alpha}\in\mb{R}^{d}$ such that for all $Q\geq 27$ and all $\bs{Q}\in[2,\infty)^{d}$ it holds {that}
\begin{equation}
S_{\bs\alpha}(Q)\ll_{\bs{\alpha}}Q(\log Q)^{2(d-1)}\log\log Q,\nonumber
\end{equation}
and, simultaneously,
\begin{equation}
S_{\bs{\alpha}^{\scriptscriptstyle{T}}}(\bs{Q})\ll_{\bs{\alpha}}(Q_{1}\dotsm Q_{d})^{d}\log(Q_{1}\dotsm Q_{d})^{d-1}\log
\log(Q_{1}\dotsm Q_{d}).\nonumber
\end{equation}
\end{cor}

\begin{proof}
Let $\phi(x):=(\log x)^{-(d-1)}(\log\log x)^{-1}$. Then for $m=1$ Condition (\ref{eq:multbad}) is equivalent to $\bs{\alpha}\in\textup{Mad}^{*}(d,
\phi^{{-1}})$, whereas for $n=1$ Condition (\ref{eq:multbad}) reduces to $\bs\alpha\in \textup{Mad}(d,
\phi^{{-1}})$. The proof follows from (\ref{eq:fregobound}) and Theorem \ref{prop:mainres}.{ Note that the conditions $Q\geq 27$ (resp., $Q_i\geq 2$) ensure that $\log\log Q$ (resp., $\log
\log(Q_{1}\dotsm Q_{d}))$ are strictly positive.}  
\end{proof}

This result is almost optimal when $(m,n)=(1,2)$ or $(m,n)=(2,1)$, in the sense that the upper bound only differs by a double logarithm factor from (\ref{eq:LeVa}). This, however, follows also from \cite[Theorem 1.3]{Fre3}, (\ref{eq:Badz})  and (\ref{eq:BadzDual}). The truly new bounds are those in dimension $d\geq 2$. These are further away from the minimal conjectural growth rate of $S_{{{Y}}}(\bs{Q})$, differing from (\ref{eq:LeVa}) by a logarithmic factor. Nonetheless, to the best of the authors' knowledge, they are new to the literature.

\subsection{A Multiplicative Version of the Dani Correspondence}
\label{subsec:Dani}

The proof of Theorem \ref{prop:mainres} rests on a new version of the 
{correspondence between Diophantine approximation and dynamics on the space of lattices}. This correspondence is streamlined specifically to study products of rational approximations and it will likely serve as a starting point to approach a variety of Diophantine problems in the multiplicative set-up. In view of this, we believe it to be of independent interest and we proceed to introduce it in this section. 

{Fix a function $\psi:[1,\infty)\to(0,1]$ 
and, for $T\geq 1$, 
define the sets
\begin{equation}
{\mc{S}_{m,n}^{\times}(
\psi,T):= \\
\left\{Y\in \mb{R}^{m\times n}:\exists\,\bs{p}\in\mb{Z}^{m},\, \bs{q}\in\mb{Z}^{n}\setminus\{\bs{0}\}\mbox{ s.t.}\begin{cases}
\prod_{i=1}^{m}|Y_{i}\bs{q}-p_{i}|<\psi(T) \\
\Pi_{+}(\bs{q})<T
\end{cases}
\right\}\nonumber}
\end{equation}
and 
\begin{equation}\label{diophlimsup}{\textup{W}_{m,n}^{\times}(\psi):=\bigcap_{T_0\geq 1}\bigcup_{T\geq 1}\mc{S}_{m,n}^{\times}(
\psi,T).}\end{equation}
Note that $\textup{W}_{m,n}^{\times}(\psi)$ is the set of matrices $Y\in \mb{R}^{m\times n}$ such that both the 
inequalities
$$
\prod_{i=1}^{m}\langle Y_{i}\bs{q}\rangle <\psi(T) \quad\mbox{and}\quad
\Pi_{+}(\bs{q})<T  
$$
have a non-trivial integer solution $\bs{q}$ for an unbounded set of parameters $T > 1$. These matrices are known in the literature as \textit{multiplicatively $\psi$-approximable}, cf.\ \cite{K23}. 
It is not hard to show (see Lemma \ref{lem:connectionwithS}) that, if one 
defines
$$\psi(x) := \frac1{xh(x)}$$
for some $h:[{1,\infty)\to[1,\infty)}$, then, for any $l\in \N$ and $\kappa > 0$
the complement of the set $\textup{Mad}\left(l,
 {h}\right)$ is contained in the set $\textup{W}_{l,1}^{\times}(
 \kappa \psi)$. Analogously, one observes that the complement of the set $\textup{Mad}^{*}\left(l,
 {h}\right)$ is contained in $\textup{W}_{1,l}^{\times}(
 \kappa \psi)$. Thus, Theorem \ref{prop:mainres} reduces to showing  that, for some appropriate function $\psi$, the complements $\textup{W}_{m,n}^{\times}(
\psi)^c$ of sets of the form \eqref{diophlimsup} (and some variations thereof) have full Hausdorff dimension.}

{The set in \eqref{diophlimsup} constinutes a multiplicative analogue of the set of \textit{$\psi$-approximable} martrices. These matrices lie at the core of the Khintchine-Groshev Theorem \cite{BV10} and can be described by means of dynamics on the moduli space of lattices, through what is known as the Dani Correspondence. This crucial connection with dynamics was developed in \cite{Dani:corr}, in the special case of well/badly approximable matrices, and in \cite{KM99} in the general case.  In 
Section \ref{sec:MultDani} of this paper, we extend this correspondence  to the multiplicative set-up by considering the multi-parameter action of a certain cone in the group of diagonal matrices in $\SL_{m+n}(\R)$. Let us be more precise. Given a matrix $Y\in \mb{R}^{m\times n}$, we define $u_Y\in \SL_{m+n}(\R)$ and the corresponding lattice $\Lambda_{Y}$ in $\R^{m+n}$ by the relation
\eq{lambday}{\Lambda_{Y}:= {u_{Y}\mb Z^{m+n}:=}\begin{pmatrix}
I_{m} & Y \\
\bs{0} & I_{n}
\end{pmatrix}\mb{Z}^{m+n}.}
For any lattice $\Lambda\subset\mb{R}^{m+n}$ we define the first minimum $\delta(\Lambda)$ of $\Lambda$ as \eq{1min}{\delta(\Lambda):= \inf_{\bs{v}\in\Lambda\setminus\{\bs{0}\}}\|\bs{v}\|,}
where $\|\cdot\|$ denotes the supremum norm. Finally, for $\bs{t}\in\R^m$ and $\bs{u}\in\R^n$ 
we set
\eq{defatu}{a(\bs{t},\bs{u}):=\textup{diag}\left(e^{t_{1}},\dotsc,e^{t_{m}},e^{-u_{1}},\dotsc,e^{-u_{n}}\right).}
It will be our standing convention that for $\bs{t}$ and $\bs{u}$ as above we will always have   \begin{equation}\label{convention}t := \sum_{i=1}^{m}t_{i}=\sum_{j=1}^{n}u_{j}.\end{equation}
Given a {continuous}
 non-increasing function  $\psi:[0,\infty)\to(0,1]$, our correspondence is then captured by the following equivalence:
 $$Y\in\textup{W}_{m,n}^{\times}(
\psi)\iff \delta\left(a(\bs{t},\bs{u})\Lambda_{Y}\right) < e^{-R(t)}
\ \mbox{for an unbounded set of vectors } (\bs{t},\bs{u})\in C_R,$$
where $C_R$ is a special cone 
inside the group of diagonal matrices in $\SL_{m+n}(\R)$ (identified with the space $\mb R^{m+n}$).
Here ${R:[0,\infty)\to \R}$ is a 
{continuous}
 function 
 associated to $\psi$ in a unique way (see \cite[Lemma 8.3]{KM99} and Lemma \ref{lem:C1}). 
We wish to highlight that the main novelty in the above statement is the observation that in order to achieve a one-to-one correspondence between multiplicative approximation and dynamics one has to let the acting cone \emph{depend on the approximating function $\psi$}  (see \eqref{eq:cone} for a precise definition). This appears to be a new observation.} The reader is referred to Section \ref{sec:MultDani} and in particular to Proposition \ref{prop:C2} for a proof of the equivalence.

\subsection{Acknowledgments}
RF is very grateful to Victor Beresnevich 
for many useful discussions, which eventually led to this paper. {Part of this work was done during DK's stay at the ETH (Z\"urich), whose hospitality is gratefully acknowledged. {We are also indebted to the anonymous referee for their suggestions, in particular, for pointing out Remark \ref{rmk:ref}}.}

\setcounter{tocdepth}{1} 
\tableofcontents

\section{Techniques and Overview of the Proof}

In this section, we briefly outline the strategy of proof and the main techniques used therein, with some emphasis on the novel ideas featuring in this paper. 
Here and throughout the paper we will use the Vinogradov and Bachmann-Landau notations. Namely, we will write $x\ll_{z} y$ for $x,y,z> 0$ to indicate that there exists a constant $c>0$, depending on the parameter $z$, such that $x\leq cy$. We will also denote by $O_{z}(x)$ an unspecified quantity $y$ such that $y\ll_{z}x$. 

To prove Theorem \ref{thm:thm1}, we argue by induction on the parameter $d$. We start by assuming that $\bs{\alpha}=(\alpha_{2},\dotsc,\alpha_{d})\in\mb{R}^{d-1}$ is a vector in (\ref{eq:mainres}) with $d$ replaced by $d-1$. We fix an interval $I\subset\mb{R}$  and 
reduce 
the problem to proving that, {for some   $\kappa>0$}, the set of $x\in I$ such that for all $S\subset\{2,\dotsc,d\}$ the inequalities
\begin{equation}\label{simult}
  {\inf_{q\ne 0}}
|q|h_{l}(|q|)\langle qx\rangle\prod_{i\in S}\langle q\alpha_{i}\rangle >\kappa
\end{equation}
{and} 
\begin{equation}\label{dual}
  {{ \inf_{\bs{q}\ne\bs{0}}\Pi_+(\bs{q})h_l\left(\Pi_+(\bs{q})\right)\left\langle q_1x + \sum_{i\in S}q_i\alpha_i
  \right\rangle>\kappa}}
\end{equation}
hold, has full Hausdorff dimension. Here, $l:=\#S+1$. The reader may refer to Section \ref{sec:Sli} for the details. We then use the multiplicative Dani Correspondence (Section \ref{sec:MultDani}) to reinterpret 
\eqref{simult} and \eqref{dual} as a condition on the the first minimum of certain lattices.

In Section \ref{sec:dansets} we introduce "dangerous" subsets of the interval $I$. These are the sets that we ultimately wish to remove from $I$. By construction, they depend on a set of indices $S\subset\{2,\dotsc,d\}$, on an integer vector 
$\binom{b_{0}}{\bs{b}}\in\mb{Z}^{l+1}$, and on a multi-time $\bs{t}\in\beta \mb{Z}^{l}$, where $\#S=l-1$ {and $\beta$ is a suitably chosen positive number}. Roughly speaking, they represent the portion of the real line where {a fixed integer vector 
is a reason for 
{\eqref{simult} and \eqref{dual}}  to fail. More precisely,} {with the notation as in \equ{lambday}--\eqref{convention} where $(m,n) = (l,1)$ and $(l,1)$ respectively},  each dangerous set has {either} the form
\begin{equation}\label{simultset}
 \left\{x\in I : \left\|a(\bs{t},t){u_{x\choose {\pi_{S}\bs{\alpha}}} {b_{0}\choose \bs{b}}}\right\|<e^{-R(t)}\right\}
\end{equation}
{or a similar  "dual" form
\begin{equation}\label{dualset}
 {
\left\{x\in I : \left\|a(t,\bs{t}){u_{(x,{\pi_{S}\bs{\alpha}^{\scriptscriptstyle{T}})}} {b_{0}\choose \bs{b}}}\right\|<e^{-R(t)}\right\},} 
\end{equation}
{where $R$ is a function of $t$ determined by $h$ and $\kappa$ via Lemma~\ref{lem:C1}.
See Definitions \ref{def:dualdan} and \ref{def:simdan}} for details.}
Our goal will be to remove all "dangerous" sets from the interval $I$ and show that the remaining points in $I$ form a full Hausdorff dimension set. This is based on ideas from \cite{Ber15}.

In Section \ref{sec:setup}, we proceed to construct a Cantor-type set contained in the complement of all dangerous intervals. To this end, we recursively subdivide the interval $I$ into smaller sub-intervals and remove those that intersect dangerous intervals defined by multi-times $\bs{t}$ that lie in a specified range. This construction is based on the work of Badziahin and Velani \cite{BV} (see also Section \ref{sec:Sli}).

To show that the Cantor-type set constructed in the previous step has full Hausdorff dimension, we must estimate how many sub-intervals are removed from $I$ at each time. We show that, under some mild assumptions, each dangerous set can be thought of as an interval. In view of this, counting the intervals that need to be removed in the construction at each time reduces to answering the following crucial question: given an interval {${J}\subset I$} and a fixed multi-time $\bs{t}$, 
\eq{countingproblem}{{\begin{aligned}\text{for how many {integer} vectors }{b_{0}\choose \bs{b}}&\\ \text{the intersection of sets \eqref{simultset} and \eqref{dualset} with }&J\text{ is nonempty?}\end{aligned}}}
\ignore{satisfy
\begin{equation}
\label{eq:Q}
\left\{x\in I : a(\bs{t},t){u_{x\choose {\pi_{S}\bs{\alpha}}} {b_{0}\choose \bs{b}}}<e^{-R(t)}\right\}\cap B\neq \varnothing?
\end{equation}}
The {objective} of Section \ref{sec:countingonav} is to answer this question.

Based on ideas of Badziahin, we show that the frequency of dangerous sets (or better intervals) at multi-time $\bs{t}$ is "on average" $1$ in $t^{l-1}$ (where $t=\sum_{i}t_{i}$). More precisely, we will prove that, if $\Delta_{\bs{t}}$ is the length of any dangerous interval at time $\bs{t}$, then there exists a number $N\geq 1$ for which any block of $N$ intervals of length $t^{l-1}\Delta_{\bs{t}}$ is intersected by at most $N$ dangerous intervals. The main task in Section \ref{sec:countingonav} will be to give a precise estimate of the number $N$. This can be reduced to a lattice-point-counting problem for the {lattices $a (\bs{t},t){\Lambda_{y\choose {\pi_{S}\bs{\alpha}}}}$  and $a(t,\bs{t}){\Lambda_{(y,{\pi_{S}\bs{\alpha}^{\scriptscriptstyle{T}})}}}$}, where $y$ is some fixed point in $I$.

The \emph{key idea} is to estimate the first minimum of the {aforementioned lattices} 
through the inductive hypothesis. {Here we use a simple but effective idea of Widmer (see \cite{Wid}), later developed by the   {first-named} author in \cite{Fre1} and \cite{Fre3}. Let us describe it for the simultaneous case;  the dual case is treated along the same lines. Note that}  any vector $\bs{v}$ in the lattice $a (\bs{t},t){\Lambda_{y\choose {\pi_{S}\bs{\alpha}}}}$ has the {form} $a (\bs{t},t){u_{y\choose {\pi_{S}\bs{\alpha}}}{b_{0}\choose \bs{b}}}$, with ${b_{0}\choose \bs{b}}$ an integer vector.  To give an estimate from below of the length of $\bs{v}$, we use the arithmetic-geometric mean inequality. Namely, we have that
\begin{equation}
\label{eq:AM-GM}
\|\bs{v}\|=\max_{i}|v_i|\geq \left(
  {\Pi_+(\bs{v})}
\right)^{1/\#\{i:v_{i}\neq 0\}}.
\end{equation}
Since $\det a (\bs{t},t)=1$, if all the components of $\bs{v}$ are non-null, we also have that
\begin{equation}
\label{eq:AM-GM1}
  {\Pi_+(\bs{v})}=|b_{0}||b_{1}+b_{0}y|\prod_{i=1}^{l}|b_{i}+b_{0}\alpha_{i}|.
\end{equation}
Now, if there are still points in the {interval ${J}$} to be removed, we can choose $y$ outside all dangerous intervals removed up to this point. This implies that the product at the right-hand side of (\ref{eq:AM-GM1}) is bounded below by the function $h_{l}^{-1}$, yielding the required estimate for the minimum. 

The strategy described in the previous paragraph works so long as the components of the vector $\bs{v}$ in (\ref{eq:AM-GM}) are all non-zero. When some of the components of $\bs{v}$ are null, we cannot rely on the fact that $\det a (\bs{t},t)=1$. To solve this problem, we proceed to estimate simultaneously the first minimum $\lambda_{1}$ of the lattice $a (\bs{t},t){\Lambda_{y\choose {\pi_{S}\bs{\alpha}}}}$ and the first minimum $\lambda_{1}^{*}$ of its dual lattice. Then we show that the product $\lambda_{1}\lambda_{1}^{*}$ can be bounded below and, hence, one of the two minima admits a favorable lower bound. Finally a result of Mahler \cite{Mah} will allow us to relate the minima of the dual lattice to those of the original lattice, thus completely solving the counting problem  \equ{countingproblem}.

Along with the multiplicative Dani correspondence, the use of the dual minimum to approach the lattice-point-counting stage of the proof is the main novelty of this paper. We wish to remark that some form of duality intrinsic to the present problem was already pointed out in \cite{Bad} (see Subsection 2.2). Here, however, we are able to make this duality fully explicit by exploiting the fact that the lattices involved in the proof of the dual case of Theorem \ref{thm:thm1} (from a Diophantine perspective) are precisely the dual (from a geometric perspective) of the lattices considered in the simultaneous case. This allows us to use both the inductive hypotheses at once. 

In Section \ref{sec:conclusion} the proof is concluded, and the constant $\kappa$ and {several  other parameters}
are chosen appropriately.

\section{Slicing and 
{Cantor-type sets}}
\label{sec:Sli}

The aim of this section is to reduce Theorem \ref{prop:mainres} to a one-dimensional statement. We start with the following easy lemma, the proof of which we leave to the reader.

\begin{lem}
\label{lem:affine}
Let $h:[0,\infty)\to \mb{R}$ be a non-decreasing sub-homogeneous function of exponent $\lambda$, i.e., such that for all $c\geq 1$ and $x\in[0,\infty)$ it holds $h(cx)\leq c^{\lambda}h(x)$. Let 
$\mu\in\mb{Q}\setminus\{0\}$ and $\bs{\nu}\in\mb{Q}^{l}$; then 
$$
{\mu\cdot\textup{Mad}\left(l,
h\right)+\bs{\nu}\subset\textup{Mad}\left(l,
h\right)\quad\mbox{and}\quad
\mu\cdot\textup{Mad}^{*}\left(l,
h\right)+\bs{\nu}\subset\textup{Mad}^{*}\left(l,
h\right).}$$
\end{lem}

{The next lemma relates the sets $\textup{Mad}\left(l,
 {h}\right) $ and $\textup{Mad}^{*}\left(l,
 {h}\right) $ to the sets \eqref{diophlimsup}  of multiplicatively $\psi$-approximable matrices for an appropriately chosen function $\psi$.}

\begin{lem}
\label{lem:connectionwithS}
{Let $\psi:[0,\infty)\to(0,1]$ be a continuous non-increasing function, and let $h:[0,\infty)\to\R$ be defined by $h(x):= \frac1{x\psi(x)}$}.
Then, 
we have that
$$
 {\textup{Mad}\left(l,
 {h}\right)=\bigcup_{\kappa>0} \textup{W}_{l,1}^{\times}(
 \kappa \psi)^c\quad\mbox{and}\quad \textup{Mad}^{*}\left(l,
  {h}\right)=\bigcup_{\kappa>0} 
 \textup{W}_{1,l}^{\times}\left(
 \kappa\psi\right)^c.}$$
\end{lem}


\begin{proof}
{Let us start by proving that $\textup{W}_{l,1}^{\times}(
 \kappa \psi)^c\subset\textup{Mad}\left(l,
 {h}\right)$ for any given $\kappa>0$. Take $\bs{\alpha}\in \mb{R}^{l\times 1}$ and suppose that the inequality
$$\prod_{i=1}^{l}|\alpha_{i}q-p_{i}|<\frac\kappa2  \psi({|q|})$$
has a solution $(\bs{p},q)\in\mb{Z}^{l+1}$ with $q\neq 0$. By changing the signs of all the components $p_i$ of $\bs{p}$ one can assume that $q > 0$.
Take  $T > q$ such that $\psi(T) \ge \psi(q)/2$. Then the inequalities $\prod_{i=1}^{l}|\alpha_{i}q-p_{i}|<\kappa\psi(T)$ and $q<T$ are satisfied.
This argument shows that for any $\bs{\alpha}\notin\textup{W}_{l,1}^{\times}(
\kappa \psi)$
there exists $M \in\N$ such that
$$q\prod_{i=1}^{l}\langle \alpha_{i}q\rangle\geq \frac\kappa2 q\psi({q})$$
for all $q \in\N$ with $q > M$. In particular this forces the product $\prod_{i=1}^{l}\langle \alpha_{i}q\rangle$ to be nonzero for all $q \in\N$, and thus for $q = 1,\dots,M$ one can write  $$q\prod_{i=1}^{l}\langle \alpha_{i}q\rangle\geq \frac{\inf_{q = 1,\dots,M}\prod_{i=1}^{l}\langle \alpha_{i}q\rangle}{\psi(1)} q\psi({q}).$$
Hence
${{\bs{\alpha}}\in \textup{Mad}\left(l,
 {h}\right).}$}
{For the other inclusion, {which is not needed for the proof of our results, 
assume 
that}
$$\bs\alpha\in\left(\bigcup_{\kappa > 0}{\textup{W}_{l,1}^{\times}(
 \kappa \psi)^c}\right)^c =\bigcap_{\kappa > 0}{\textup{W}_{l,1}^{\times}(
 \kappa \psi)}.$$
 Then, for any $\kappa>0$ the system
\begin{equation}
\label{eq:nnnnew}
\begin{cases}
\prod_{i=1}^{l}\langle \alpha_{i}{q}\rangle <\kappa\cdot \psi(T)\leq \kappa\cdot \psi(q) \\
|q|<T
\end{cases}   
\end{equation}
has a solution $q\neq 0$ for at least one (in fact, infinitely many) $T\in\mb N$. Let $\kappa_n$ be a sequence of values of $\kappa$ tending to $0$, and let $T_n$ be one fixed corresponding value of $T$ for which the system (\ref{eq:nnnnew}) has a solution $q_n$. Then, {we must have} $q_n\to\infty$. As above, we may assume that $q_n>0$. Then, it follows from (\ref{eq:nnnnew}) that
$$\lim_{n\to\infty}|q_n|\prod_{i=1}^{l}\langle \alpha_{i}q\rangle<\kappa_n\cdot q_n\psi(q_n)=\frac{\kappa_n}{h(q_n)},$$
whence {$\bs \alpha\notin \textup{Mad}\left(l,h\right)$}. 
}
The 
second equality 
is 
proved similarly.
\end{proof}

In view of Lemmas \ref{lem:affine} and \ref{lem:connectionwithS}, {Theorem \ref{prop:mainres}  can be reduced} to 
the subsequent proposition.

\begin{prop}
\label{prop:reduction1}
Let \eq{psil}{  {\psi_{l}(x):=x^{-1}h_{l}(x)^{-1} = \frac1{x(\log^{+}x)^{l-1}\log^{+}\log^{+}x}}} for $l=1,\dotsc,d$. Then, there exist a box $B\subset\mb{R}^{d}$ and a constant $\kappa=\kappa (d,B)>0$ such that the set
$${\bigcap_{l=1}^{d}\bigcap_{\#S=l}\pi_{S}^{-1}\left(\textup{W}_{l,1}^{\times}(
\kappa \psi_{l})^{c}\cap\textup{W}_{1,l}^{\times}(
\kappa \psi_{l})^{c}\right) \cap B}$$
has full Hausdorff dimension. 
\end{prop}

We will prove Proposition \ref{prop:reduction1} by induction on $d$. For $d=1$, the proof is analogous to that for $d>1$, but no inductive hypothesis is required. More on this can be found at the beginning of Section \ref{sec:setup}. {Now, given $d\ge 2$}, by the inductive hypothesis and Lemma \ref{lem:affine} we can find a full Hausdorff dimension set of vectors $(\alpha_{2},\dotsc,\alpha_{d})\in[0,1]^{d-1}$ such that for 
some fixed constant ${\gamma}>0$, both the following conditions hold:
\begin{equation}
\label{eq:multinductivesim}\begin{aligned}
{|q|}\prod_{i\in S}\langle q\alpha_{i}\rangle \geq {\gamma}h_{l-1}({|q|})^{-1}\\  {\text{for all nonempty }  S\subset\{2,\dotsc,d\}\text{ with }\#S=l-1}&{\text{ and all }}q\in\mb{Z}\setminus\{0\},\end{aligned}
\end{equation}
 and
\begin{equation}
\label{eq:multinductivedual}
\begin{aligned}
  {\Pi_+(\bs{q})}\left\langle\bs{q}\cdot \pi_{S}(\bs{\alpha})\right\rangle\geq {\gamma}h_{l-1}\left(
  {\Pi_+(\bs{q})}\right)^{-1}\\\  {\text{for all nonempty }  S\subset\{2,\dotsc,d\}\text{ with }\#S=l-1\text{ and all }}&\bs{q}\in\mb{Z}^{l-1}\setminus\{\bs{0}\}.\end{aligned}
\end{equation}

To carry out the inductive step, we will use the following well-known "slicing" lemma (see \cite[Corollary 7.12]{Fal}).

\begin{lem}[Marstrand Slicing Lemma]
\label{lem:Marstrand}
Let $d>1$, let $A\subset \mb{R}^{d}$, and let $U\subset\mb{R}^{d-1}$. If for all $\bs{u}\in U$
$$\dim\{t\in\mb{R}:(t,\bs{u})\in A\}\geq s>0,\vspace{2mm}$$
then $\dim A\geq \dim U+s$, where $\dim$ denotes the Hausdorff dimension.
\end{lem}

{For fixed $(\alpha_{2},\dotsc,\alpha_{d})\in[0,1]^{d-1}$,} let $\bs{f}:\mb{R}\to\mb{R}^{d}$ be the function \eq{deff}{\bs{f}(x):=(x,\alpha_{2},\dotsc,\alpha_{d}).}
In view of Lemma \ref{lem:Marstrand}, {the inductive step in the proof of} Proposition \ref{prop:reduction1} can 
{proceed as follows.} One fixes {$(\alpha_{2},\dotsc,\alpha_{d})\in[0,1]^{d-1}$ {such that  \eqref{eq:multinductivesim} and \eqref{eq:multinductivedual} hold 
with some ${\gamma}>0$}, and aims to find an interval $I\subset\mb{R}$ and a constant $\kappa>0$ such that the set
\begin{equation*}
{\bigcap_{l=2}^{d}\ \bigcap_{S\subset\{2,\dotsc,d\},\,\#S=l-1}
\left\{x\in I : \pi_{\{1\}\cup S}\bs{f}(x)\in \textup{W}_{l,1}^{\times}(\kappa \psi_{l})^{c}\cap \textup{W}_{1,l}^{\times}(\kappa \psi_{l})^{c}\right\},}
\end{equation*}
 has full Hausdorff dimension, {where $\bs{f}$ is as in \equ{deff}}.

{In order to carry out the plan described above, let us recall the definition and properties of Cantor-type sets introduced in \cite{BV}. 
The notation below is borrowed from \cite[\S5]{Ber15}.}

{Given a collection $\mc{I}$ of compact intervals in $\R$ and $r\in\N$, let $\frac1
r \mc{I}$
denote the collection of intervals obtained by dividing each interval in $\mc{I}$ into $r$ equal
closed subintervals.} Let ${\{r_{k}\}}$ be a sequence of positive 
{natural numbers}. We call a sequence ${\{\mc{I}_{k}\}}$ of interval collections in $\mb{R}$ an $r_{k}$-sequence if $\mc{I}_{k+1}\subset r_{k}^{-1}\mc{I}_{k}$ for all $k=0,1,\dotsc$. We define
$$\hat{\mc{I}}_{k}:=\frac{1}{r_{k-1}}\mc{I}_{k-1}\setminus\mc{I}_{k}$$
and the \textit{Cantor-type set} associated to $\mc{I}_{k}$ as
$$\mc{K}(\mc{I}_{k}):=\bigcap_{n\geq 0}\bigcup_{I\in\mc{I}_{k}}I.$$
Any set constructed through this procedure is called an $r_{k}$-Cantor-type set.

For an interval $J\subset\mb{R}$ and a collection of intervals $\mc{I}'$ in $\mb{R}$ we set
$$\mc{I}'\sqcap J:=\{I\in\mc{I}':I\subset J\}.$$
We define the $k$-th local characteristic of the sequence $\mc{I}_{k}$ as
\begin{equation}
\label{eq:localchar}
{\Delta}_{k}:=\min_{\left\{\hat{\mc{I}}_{k,p}\right\}}\sum_{p=0}^{n-1}\left(\prod_{i=p}^{n-1}\frac{4}{r_{i}}\right)\max_{I_{p}\in\mc{I}_{p}}\#\hat{\mc{I}}_{k,p}\sqcap I_{p},
\end{equation}
where $\left\{\hat{\mc{I}}_{k,p}\right\}$ varies through the partitions of the collection $\hat{\mc{I}}_{k}$ into $k$ subsets ($p=0,\dotsc,k-1$). Moreover, we define the global characteristic of the sequence $\{\mc{I}_{k}\}$ as
$${\Delta}:=\sup_{k\geq 0}{\Delta}_{k}.$$
Then, we have the following.

\begin{definition}
\label{def:Cantorrich}
A set $A\subset\mb{R}$ is said to be $r_{k}$-Cantor-rich if for any $\varepsilon>0$ there exists an $r_{k}$-Cantor-type set $\mc{K}(\mc{I}_{k})\subset A$ such that $\mc{I}_{k}$ has global characteristic ${\Delta}<\varepsilon$. 
\end{definition}
{The importance of  Cantor-rich sets is due to their nice intersection properties: according to \cite[Theorem 5]{BV}, the intersection of any finite number of $r_{k}$-Cantor-rich sets with same initial interval collection $\mc{I}_{0}$ is $r_{k}$-Cantor-rich. Furthermore, if  $\mc{K}(\mc{I}_{k})\subset\mb{R}$ is an $r_{k}$-Cantor-type set such that the global characteristic of $\mc{I}_{k}$ is less or equal to $1$, then
$$\dim\mc{K}(\mc{I}_{k})\geq\liminf_{k\to \infty}1-\frac{\log 2}{\log r_{k}},$$
see \cite[Theorem 4]{BV}.}
The two results stated above imply the following fact:
 
\begin{theorem}
\label{cor:CRFullHaus}
Let $r_{k}$ be a sequence of 
{natural} numbers tending to $\infty$. Then the intersection of a finite number of $r_{k}$-Cantor-rich sets with same initial interval collection $\mc{I}_{0}$ has full Hausdorff dimension. 
\end{theorem}

{In view of 
 the above discussion, the following statement will suffice for the inductive step and thus will imply Proposition \ref{prop:reduction1}:
\begin{prop}
\label{prop:fibrestatement} Let $d\ge 2$, 
take {$(\alpha_{2},\dotsc,\alpha_{d})\in[0,1]^{d-1}$, and suppose that  ${\gamma}>0$ is} a constant for which \eqref{eq:multinductivesim} and \eqref{eq:multinductivedual} hold. Let  $\bs{f}$ be as in \equ{deff}. 
Then,  there exist an interval $I=I({\gamma})\subset\mb{R}$, a constant $0<\kappa=\kappa({\gamma},I)<1$, and a sequence of 
{natural} numbers $r_{k}
$ with $r_{k}\to \infty $ such that  for any $
S \subset\{2,\dotsc,d\}$ with $\#S=l-1$ ($1\leq l\leq d$),
the set
\begin{equation}\label{mainset}
\left\{x\in I : \pi_{\{1\}\cup S}\bs{f}(x)\in \textup{W}_{l,1}^{\times}(\kappa \psi_{l})^{c}\cap \textup{W}_{1,l}^{\times}(\kappa \psi_{l})^{c}\right\},
\end{equation}
is  $r_{k}$-Cantor-rich.
\end{prop}}

\section{The Multiplicative Dani Correspondence}
\label{sec:MultDani}

Let us start {this section} with a historical interlude. For $m,n\in\N$ 
consider
the subgroup $\{a_t\}$ of $\SL_{m+n}(\R)$, where
\begin{equation} \label{eq:gt-mn}
a_t = \text{\rm diag}(\underbrace{e^{t/m},\dots,e^{t/m}}_{\text{$m$ times}},\underbrace{e^{-t/n},\dots,e^{-t/n}}_{\text{$n$ times}})\,.
\end{equation}
A connection between the behavior of certain $a_t$-trajectories in the space of unimodular lattices in $\R^{m+n}$ and simultaneous Diophantine approximation was implicitly observed by Davenport and Schmidt \cite{DS} in the late 1960s, and explicitly written down by Dani in 1985 \cite{Dani:corr}. Later, this connection, in a more general form, was called "Dani Correspondence" \cite{KM99}.
The next lemma 
is a special case of \cite[Lemma 8.3]{KM99}, which has repeatedly been used in the past to set-up the correspondence between dynamics and Diophantine approximation.

\begin{lem}
\label{lem:C1}  
For any $m,n\in\N$ 
and  any  {continuous}
 non-increasing function \linebreak $\psi:{(}0,\infty)\to(0,1]$ there exists a unique 
{continuous}
 function 
 $R:[0,\infty)\to \R$  
such that %
\eq{incr}{\text{the map }t\mapsto t-nR(t) \text{ is strictly increasing and tends to $\infty$ as } t\to\infty,}
\eq{nondecr}{\text{the map }t\mapsto t+mR(t)\text{ is non-decreasing,}}
and 
\eq{corr}{\psi\left(e^{t-nR(t)}\right)=e^{-t-mR(t)}\quad\forall\,t\geq 0.}
Conversely, given 
a 
 {continuous}
function $R:[{0},\infty)\to \mb{R}$ {with $R(0) \ge 0$} satisfying 
\equ{incr} 
and 
\equ{nondecr},  
there exists a unique 
 {continuous}
    non-increasing function $\psi:[{e^{- nR(0)}},\infty)\to(0,{1]}$
    such that \equ{corr} holds. 
\ignore{For any smooth non-increasing function $\psi:[0,\infty)\to(0,1]$ there exists a unique smooth function $R:[0,\infty)\to \mb{R}$, such that $R(0)\geq 0$, the map
$$t\mapsto t-nR(t)$$
is strictly increasing, the map
$$t\mapsto t+mR(t)$$
is non-decreasing, and for all $t\geq 0$ it holds
$$\psi\left(e^{t-nR(t)}\right)=e^{-t-mR(t)}.$$
Conversely, given a smooth function $R:[0,\infty)\to \mb{R}$ such that the map $t\mapsto t-nR(t)$ is strictly increasing and the map $t\mapsto t+mR(t)$ is non-decreasing, there exists a unique smooth non-increasing function $\psi:[x_{0},\infty)\to(0,\infty)$, with $x_{0}=e^{-nR(0)}$, such that for all $t\geq 0$ it holds
$$\psi\left(e^{t-nR(t)}\right)=e^{-t-mR(t)}$$.}
\end{lem}


\begin{rmk}\label{rmk:4.2} \rm {
More generally, \cite[Lemma 8.3]{KM99} established a correspondence between functions 
$\psi:[x_0,\infty)\to (0,\infty)$ and   $R:[t_0,\infty)\to \R$, with $x_0,t_0\in\mb R$. Lemma \ref{lem:C1} can be obtained by letting $t_0=0$ and assuming that $\psi(x) \le 1$ for all $x$.} Further, observe that \equ{corr} and the condition that the image of $\psi$ is contained in the interval $(0,1]$ imply that $R(0)\geq 0$; {and, conversely, the assumption $R(0)\geq 0$, in view of \equ{corr}, implies that $\psi (e^{-nR(0)}) = e^{-mR(0)} \le 1$}. Note also that in \cite{KM99}, just the continuity of the functions $\psi$ and $R$ was assumed; however, it is easy to see that the smoothness (perhaps on some range) of one function follows easily from that of the other.\end{rmk}

\begin{rmk}
{In order to apply Lemma \ref{lem:C1}, we are going to adopt the following convention: whenever the function $\psi$ is defined for sufficiently large values of the argument (say for $x\ge x_0$), it will be extended by default to the whole positive half-line $(0,\infty)$ by letting $\psi(x) := \psi(x_0)$ for $0 < x < x_0$.}    
\end{rmk}

To state the standard form of the correspondence between approximation and dynamics on the space of lattices, for $\psi$ as above 
and $T\geq 1$ let us define
\begin{equation}
{\mc{S}_{m,n}(
\psi,T):= \\
\left\{Y\in \mb{R}^{m\times n}:\exists\, \bs{p}\in\mb{Z}^{m},\, \bs{q}\in\mb{Z}^{n}\setminus\{\bs{0}\}\mbox{ s.t.}\begin{cases}
{\|Y\bs{q}- \bs{p}}\|^m<\psi(T) \\
{\|\bs{q}}\|^n<T
\end{cases}
\right\}\nonumber}.
\end{equation}
 Note that $\mc{S}_{m,n}(
\psi,T)$ is contained in the set $\mc{S}_{m,n}^\times(
\psi,T)$ defined in the {Introduction. Consider the set} 
\begin{equation}\label{diophlimsupclassical}{\textup{W}_{m,n}(\psi):=\bigcap_{T_0\geq 1}\bigcup_{T\geq 1}\mc{S}_{m,n}(
\psi,T)}\nonumber\end{equation}
of $\psi$-approximable matrices. {Recall the definitions \equ{lambday} and  \equ{1min} of $\Lambda_Y$ and $\delta(\Lambda)$.}
\ignore{Given a matrix $Y\in \mb{R}^{m\times n}$, {generalizing \equ{lambdaalpha}},  define \eq{lambday}{\Lambda_{Y}:= {u_{Y}\mb Z^{m+n}:=}\begin{pmatrix}
I_{m} & Y \\
\bs{0} & I_{n}
\end{pmatrix}\mb{Z}^{m+n}.}
{Also} let \eq{1min}{\delta(\Lambda):= \inf_{\bs{v}\in\Lambda\setminus\{\bs{0}\}}|\bs{v}|} denote the first minimum of a lattice $\Lambda\subset\mb{R}^{m+n}$ with respect to the supremum norm.} Then we have the following statement, which is a variation on \cite[Theorem 8.5]{KM99} (we omit the proof since it can be easily reconstructed from the proof of its multiplicative analog, Proposition \ref{prop:C2}):
\begin{prop}
\label{prop:C1}
Let 
$\psi:{(0},\infty)\to(0,1]$ be a  {continuous}
 non-increasing function,  and let $R$ be the function corresponding to $\psi$ via Lemma  \ref{lem:C1}.
 Take $Y\in \mb{R}^{m\times n}$ and $T\geq 1$. Then $Y\in\mc{S}_{m,n} (
\psi,T)$ if and only if 
\begin{equation}
\label{eq:deltastandard}
\delta (a_t\Lambda_{Y}) < e^{-R(t)},
\end{equation}
where $t\ge 0$ is defined by \eq{Tt}{T=e^{t-nR(t)}.} {Consequently $Y\in\textup{W}_{m,n} (
\psi)$ if and only if \eqref{eq:deltastandard} is satisfied for an unbounded set of $t \ge 0$.}
\end{prop} 
{For $s > 0$ 
define $\varphi_s(x):=x^{-s}$}. A classical case is that of $\psi = c\varphi_1$, where $0< c \le 1$. In this setting, \equ{corr} gives $e^{-R(t)} = c^{\frac1{m+n}}$ {for all $t \ge 0$} --- a constant function. Now, recall that a matrix $Y$ is said to be \textit{badly approximable} if 
there exists $c> 0$ such that $Y$ is not in the set $\mc{S}_{m,n}(
c\varphi_1,T)$
for all $T\ge 1$. This, via Proposition \ref{prop:C1}, translates into the condition $$\delta (a_t\Lambda_{Y})\geq c^{\frac1{m+n}}\text{ for some $c>0$ and all }
{t\ge 0}.$$ In view of Mahler's Compactness Criterion (see e.g.\  \cite{Cas}) this is equivalent to  the $a_t$-trajectory of $\Lambda_{Y}$ being bounded, 
{cf.\ \cite[Theorem 2.20]{Dani:corr}}.

The goal of this section is to extend the  correspondence described above to the multiplicative set-up. Such extensions have been considered before but only for some special cases, see \cite{KLW, KM98, KM99}. We are going to state a precise and most general multiplicative analog of Proposition \ref{prop:C1}. As in {the aforementioned} papers, this is done by considering the multi-parameter action of a certain cone in the group of diagonal matrices. The new ingredient, however, is the observation that in order to achieve a one-to-one correspondence between multiplicative approximation and dynamics, one has to adjust the acting cone based on  the approximating function.

{Recall that for $\bs{t}\in\R^m$ and $\bs{u}\in\R^n$ 
we defined a diagonal matrix $a(\bs{t},\bs{u})\in\SL_{m+n}(\R)$ via
\equ{defatu} 
{with the convention that} $t := \sum_{i=1}^{m}t_{i}=\sum_{j=1}^{n}u_{j}$ as in  \eqref{convention}.}
For
a function $R:[{0,\infty)\to \R}$ we set
\begin{equation}\label{eq:cone}
C_{R}:=  
\big\{(\bs{t},\bs{u}):\bs{t}\in\mb{R}^{m},\, \bs{u}\in\mb{R}^{n},\ {t\ge 0},\ 
t_{i}> -R(t),\, u_{j}>R(t)
\ \forall\,i,j\big\}.
\end{equation}
Then  the following result holds.

\begin{prop}
\label{prop:C2}
Let 
$\psi:{(0},\infty)\to(0,1]$ be a  {continuous}
 non-increasing function, and let $R$ be the function corresponding to $\psi$ via Lemma  \ref{lem:C1}.
 Take $Y\in \mb{R}^{m\times n}$ and $T\geq 1$. Then
$Y\in\mc{S}_{m,n}^{\times}(
\psi,T)$ if and only there exists a vector $(\bs{t},\bs{u})\in C_{R}$, with ${t}
$ defined by \equ{Tt}, such that
\begin{equation}
\label{eq:deltaC}
\delta \left(a(\bs{t},\bs{u})\Lambda_{Y}\right) < e^{-R(t)}.
\end{equation}
{Consequently $Y\in\textup{W}_{m,n}^{\times} (
\psi)$ if and only if \eqref{eq:deltaC} is satisfied for an unbounded set of $(\bs{t},\bs{u})\in C_{R}$.}
\end{prop}


\begin{proof}
Let us fix $T\geq 1$, $Y\in\mc{S}_{m,n}^{\times}(\psi,T)$, and let us pick $t \ge 0
$ such that $T=e^{t-nR(t)}$ (recall that $R(0)\geq 0$ and the function $t\mapsto t-nR(t)$ is strictly increasing). We start by noticing that, if $(\bs{p},\bs{q})$ is a non-trivial solution to 
\begin{equation}
\label{eq:1l2}
\prod_{i=1}^{m}|Y_{i}\bs{q}-p_{i}|<\psi(T)=\psi\left(e^{t-nR(t)}\right)=e^{-t-mR(t)},
\end{equation}
there will be a non-trivial solution $(\bs{p}',\bs{q})$ to (\ref{eq:1l2}) such that $|Y_{i}\bs{q}-p_{i}'|<1$ for all $i$. We can therefore assume that $(\bs{p},\bs{q})$ has this property. Hence, for $i=1,\dotsc,m$ we can find numbers $t_{i}> -R(t)$ (potentially infinite) such that
\eq{equality}{|Y_{i}\bs{q}-p_{i}|=e^{-t_{i}-R(t)}\text{ for all }i=1,\dotsc,m.}
  {Inequality (\ref{eq:1l2}) 
then implies that 
$\sum_{i}t_{i}> t$} (with the convention that $\sum_{i}t_{i}=\infty$ if some of the parameters $t_{i}$ are infinite).   {Then one can decrease all the parameters $t_{i}$ in such a way that the vector $(t_1,\dots,t_m)$ is still in the cone 
$\big\{t_{i}>-R(t),\ i=1,\dotsc,m\big\}$, 
and at the same time $\sum_{i}t_{i}=t$. This way all the equalities in \equ{equality} will turn into strict inequalities, that is, we have}
\begin{equation}
\label{eq:max1}
\max_{i}e^{t_{i}}|Y_{i}\bs{q}-p_{i}|<e^{-R(t)}.
\end{equation}
Further, we observe that, by our assumption, it holds that
$$\Pi_{+}(\bs{q})
<T=e^{t-nR(t)}.$$
Hence, for $j=1,\dotsc,n$ we can find $u_{j}\geq R(t)$ such that
$${|q_{j}|_{+}}=e^{u_{j}-R(t)}.$$
  {The two inequalities above imply that $\sum_{j}u_{j} < t$. Therefore, by increasing all the parameters $u_{j}$,} we can assume that $\sum_{j}u_{j}=t$, that $u_{j}>R(t)$ for $j=1,\dotsc,n$, and that
\begin{equation}
\label{eq:max2}
\max_{j}e^{-u_{j}}{|q_{j}|_{+}}<e^{-R(t)}.
\end{equation}
Now, (\ref{eq:max1}) and (\ref{eq:max2}) imply (\ref{eq:deltaC}), concluding the proof of this implication.

On the other hand, assume that (\ref{eq:deltaC}) holds. Then for $i=1,\dotsc,m$ we have
$$e^{t_{i}}|Y_{i}\bs{q}-p_{i}|<e^{-R(t)},$$
whence
$$\prod_{i=1}^{m}|Y_{i}\bs{q}-p_{i}|<e^{-t-mR(t)}=\psi(T).$$
Moreover, for $j=1,\dotsc,n$ we have
$e^{-u_{j}}q_{j}<e^{-R(t)},$
  {or, equivalently,}
$$q_{j}<e^{u_{j}-R(t)},$$
  {which, since $u_{j}-R(t)>0$, 
can be strengthened to}
$${|q_{j}|_{+}}<e^{u_{j}-R(t)}.$$
Then, by multiplying these inequalities for $j=1,\dotsc,n$, we obtain
$$\prod_{j=1}^{n}{|q_{j}|_{+}}<e^{t-nR(t)}=T,$$
concluding the proof. 
{The second assertion of the proposition follows trivially.}
\end{proof}

\begin{rmk}\label{rmk:EKL} \rm   {We point out that the novelty of the correspondence presented in the 
  above proposition is the appearance of the cone ${C}_R$ which depends on $R$, and thus implicitly on $\psi$. 
Previous  versions of this correspondence were utilizing the cone 
$$C_{0} =  
\left\{(\bs{t},\bs{u}):\bs{t}\in\mb{R}^{m},\, \bs{u}\in\mb{R}^{n},\ \sum_{i=1}^{m}t_{i}=\sum_{j=1}^{n}u_{j} 
\mbox{ and }t_{i}> 0,\, u_{j}>0
\ \forall\,i,j\right\}.$$ 
And indeed, in some special cases the correspondence can be reduced to $C_0$.} {For example,
it 
is proved by Einsiedler, Katok and Lindenstrauss 
\cite[Proposition 11.1]{EKL}
that
$Y\in\R^{2\times 1}$ is multiplicatively badly approximable if and only if $$
\inf_{(\bs{t},\bs{u})\in C_{0}}\delta \left(a(\bs{t},\bs{u})\Lambda_{Y}\right) > 0;$$
that is, if and only if the $a(C_0)$-trajectory of $\Lambda_{Y}$ is bounded. This equivalence 
was used in \cite{EKL} as a justification for a reduction of 
Littlewood's Conjecture to a statement about bounded orbits in the space of lattices in $\R^3$. In fact 
\cite[Proposition 11.1]{EKL} can be extended to arbitrary $m,n\in \N$. We direct the reader to the forthcoming paper \cite{BFK} for more details on this. See also \cite[Corollary 2.2]{KM98} and \cite[Proposition 3.1]{KLW} for partial results on the aforementioned correspondence in the case when $\psi = \varphi_s$  with $s > 1$. This case is related to matrices known as \emph{very well multiplicatively approximable}.}
\end{rmk}

We now state and prove a discrete version of Proposition \ref{prop:C2}, which will be useful for our computations.

\begin{cor}
\label{cor:C3}
Let $Y\in \mb{R}^{m\times n}$, and let $\beta\geq 1$ be a fixed parameter. Let $\psi:(0,\infty)\to(0,1]$ be a non-increasing function such that $\psi(cx)\gg_{m,n,\beta} c^{-\lambda}\psi(x)$ for all $x\geq 0$ and all $c\geq 1$, for a given $\lambda\geq 1$ independent of $c$. Let $R$ be the function corresponding to $\psi$ through Lemma \ref{lem:C1} and assume that
\eq{deltabound}{\delta \left(a(\bs{t},\bs{u})\Lambda_{Y}\right)> e^{-R(t)}}
for all $(\bs{t},\bs{u})\in C_{R}\cap \beta\mb{Z}^{m+n}$.
Then there exists a constant $c'>0$ only depending on $m,n,\beta$, and $\lambda$ such that $Y\notin 
\textup{W}_{m,n}^{\times}(
c'\psi)$.
\end{cor}

\begin{proof}
{Take $(\bs{t}',\bs{u}')\in C_{R}$ and let $t' :=\sum_{i=1}^{m}t'_{i}=\sum_{j=1}^{n}u'_{j}$. Let us assume first that there exists $j_0\in\{1,\dotsc,n\}$ with $u_{j_0}'> R(t')+(2m+n)\beta$. In this case, we increase all the components $t_i'$ for $i=1,\dotsc,m$ by a quantity between $0$ and $\beta$ and all the components $u_j'$ for $j\neq j_0$ by a quantity between $(m/n)\beta$ and $(m/n+1)\beta$ in order to make them integer multiples of $\beta$. We call these new components $t_i$ for $i=1,\dotsc,m$ and $u_j$ for $j\neq j_0 $. {
In addition}, we decrease (or increase) the component $u_{j_0}'$ by a quantity between $0$ and $(2m+n-1)\beta$ to obtain $u_{j_0}\in\beta\mb Z$ in such a way that
\begin{equation}
\label{eq:betaboundnew}
t'\leq \sum_{i}t_i=\sum_{j}u_j<t'+m\beta.
\end{equation}
Then we have
\begin{equation}
\label{eq:nnew}
t' \le t\underset{\equ{nondecr}}\Longrightarrow R(t) \le R(t') + \frac{t-t'}n \le  R(t') + \frac{m\beta}{n}.
\end{equation}
It follows that $(\bs t,\bs u)\in C_R$, since $u_j\geq u_{j}'+m\beta/n> R(t')+m\beta/n\geq R(t)$ for $j\neq j_0$ and $u_{j_0}\geq u_{j_0}'-(2m+n-1)\beta> R(t')+\beta\geq R(t)$. Observe that, by construction, we have
\begin{equation}
\label{eq:betabound}
{\|(\bs t,\bs u)-(\bs t',\bs u')\|}\leq (2m+n-1)\beta.
\end{equation}
In view of \equ{deltabound}, (\ref{eq:nnew}), and (\ref{eq:betabound}), we deduce that
 \begin{equation}
 \label{eq:nnnew}
 \delta \left(a(\bs{t}',\bs{u}')\Lambda_{Y}\right)> 
 e^{-R(t') - C},    
 \end{equation}
 where $C :=\beta\left(2m+n\right)$. Note that if $u_{j}'\leq R(t')+(2m+n)\beta$ for $j=1,\dotsc,n$, Equation (\ref{eq:nnnew}) is equally true. Then, from Proposition \ref{prop:C2}  it follows that for all $T\geq 1$ we have that $Y\notin \mc{S}_{m,n}^{\times}(
\tilde{\psi},T)$, where $\tilde{\psi}$ is the function corresponding to $\tilde{R}:=R+C$ through Lemma \ref{lem:C1}. In fact, it is easy to see  from \equ{corr} that $\tilde{\psi}$ is given by the formula $$\tilde{\psi}(x) = e^{-mC}\psi(e^{-nC}x).$$ Thus, from  the hypothesis on $\psi$ we deduce that $\tilde{\psi}\geq c'\psi$ for some $c'=c'(m,n,\beta,\lambda)>0$, and, consequently, $Y\notin \mc{S}_{m,n}^{\times}(
c'\psi,T)$ for all $T\geq 1$, concluding the proof.}
\end{proof}

{For technical reasons, from now on we will be working with the function
$$\psi(x):=\min\{1,\kappa\psi_{l,\beta}(x)\}$$}
for $x>0$, where $0<\kappa<1$,
$$\psi_{l,\beta}{(x)}:=\frac{1}{xh_{l,\beta}(x)} = \begin{cases}\dfrac{1}{x\beta^{l-1}\log \beta}\quad \text{ if } x\le e^\beta\\[4mm]
\dfrac1{x(\log x)^{l-1}\log\log x}\ \quad \text{ if } x > e^\beta\end{cases},$$
and
$$h_{l,\beta}(x):=h_{l}\left(\max\{x,e^{\beta}\}\right)=\left(\log\max\{x,e^{\beta}\}\right)^{l-1}\log\log\max\{x,e^{\beta}\}$$ for a fixed stretching parameter $\beta\geq e$ (see Corollary \ref{cor:C3}). Note that {the function $\psi_l$ defined in \equ{psil} coincides with $\psi_{l,0}$ for all $x\geq 1$, {and that the restriction $\beta\geq e$ ensures that $\psi_{l,\beta}{(x)} < 1$ for all $x\geq 1$}. It is easy to see that,} once Proposition \ref{prop:fibrestatement} is proved for the function $\psi$ {in place of $\kappa \psi_l$}, it will be enough to replace $\psi$ by $\kappa'\psi_l$, where $\kappa'=\kappa/\beta^{d}$, to prove the original version of Proposition \ref{prop:fibrestatement}.

We conclude this section by highlighting some helpful properties of the function $R$ corresponding to $\psi$ through Lemma \ref{lem:C1}. Henceforth, the constant $\kappa$ will always be assumed to be between $0$ and $1$.

\begin{lem}
\label{lem:ridof1}
{Let $\psi(x):=\min\{1,\kappa\psi_{l,\beta}(x)\}$ {and let $R$ be the function corresponding to $\psi$ through Lemma \ref{lem:C1} with} $(m,n)=(1,l)$ {or} $(m,n)=(l,1)$. Then for all $t\geq 0$ it holds that
$$\psi\left(e^{t-nR(t)}\right)=\kappa\psi_{l,\beta}\left(e^{t-nR(t)}\right).$$}
\end{lem}

\begin{proof}
By Lemma \ref{lem:C1},
\begin{equation}
\label{eq:xx}
\psi\left(e^{t-nR(t)}\right)=e^{-t-mR(t)}    
\end{equation}
for all $t\geq 0$. Now, if it were $R(0)=0$, \eqref{eq:xx} would imply that $\kappa \psi_{l,\beta}(1)\geq\psi(1)=1$, which contradicts our definition of $\psi_{l,\beta}$. Thus, it must be $R(0)>0$ and hence
$$\psi\left(e^{-nR(0)}\right)=e^{-mR(0)}<1.$$
Since the map $t\mapsto t+mR(t)$ is non-decreasing, it follows from \eqref{eq:xx} that
$$\psi\left(e^{t-nR(t)}\right)<1$$
for all $t\geq 0$. This shows the desired equality.
\end{proof}

\begin{lem}\label{lem:additional}
Let ${\psi(x):=\min\{1,\kappa \psi_{l,\beta}(x)}\}$,
{and let $R$ be the function corresponding to $\psi$ through Lemma \ref{lem:C1} with} $(m,n)=(1,l)$ {or} $(m,n)=(l,1)$.
Then we have that
\begin{equation}
\label{eq:123}
e^{(l+1)R(t)}=\kappa^{-1}h_{l,\beta}\left(e^{t-nR(t)}\right)
\end{equation}
and
\begin{equation}
\label{eq:111}
\kappa^{-1}\leq e^{(l+1)R(t)}\leq\kappa^{-1}\max\{t,\beta\}^{l-1}\log\max\{t,\beta\} 
\end{equation}
for all $t\geq 0$. In particular, for all values of $l$ the function $t\mapsto R(t)$ is non-decreasing. Finally, for $t\geq 4\max\{|\log\kappa|,l\log\beta\}$  we have that  
\begin{equation}
\label{eq:222}
e^{(l+1)R(t)}\geq\frac{\kappa^{-1}}{2^{l}}\max\{t,\beta\}^{l-1}\log\max\{t,\beta\}.
\end{equation}
\end{lem}

\begin{proof}
By Lemma \ref{lem:C1} and Lemma \ref{lem:ridof1}, we have that
$$\kappa e^{-t+nR(t)}h_{l,\beta}\left(e^{t-nR(t)}\right)=\psi\left(e^{t-nR(t)}\right)=e^{-t-mR(t)}$$
for all  $t\geq 0$, whence
\begin{equation}
\label{eq:123'}
e^{(m+n)R(t)}=\kappa^{-1} h_{l,\beta}\left(e^{t-nR(t)}\right),
\end{equation}
which implies (\ref{eq:123}) and the left-hand side of (\ref{eq:111}). Moreover, since the function $t\mapsto t-nR(t)$ is strictly increasing, it follows from (\ref{eq:123'}) that the function $t\mapsto R(t)$ is non-decreasing. Since $R(0)\geq 0$ and $R$ and $h_{l,\beta}$ are non-decreasing functions, from (\ref{eq:123'}) we also deduce that
$$e^{(l+1)R(t)}=\kappa^{-1} h_{l,\beta}\left(e^{t-nR(t)}\right)\leq \kappa^{-1}h_{l,\beta}(e^{t})\leq \kappa^{-1}\max\{t,\beta\}^{l-1}\log\max\{t,\beta\},$$
proving (\ref{eq:111}). Let us now assume that $t\geq 4\max\{|\log\kappa|,l\log\beta\}$. By taking logarithms on both sides of the upper bound in (\ref{eq:111}), we have that
\begin{multline}
\label{eq:tvsRt}
(l+1)R(t)\leq |\log\kappa|+(l-1)\log\max\{t,\beta\}+\log^{+}\log\max\{t,\beta\}\\
\leq |\log\kappa|+l\log\max\{t,\beta\}\leq |\log\kappa|+\max\{\log^{+}t,l\log\beta\}\leq t/2.
\end{multline}
Then, from (\ref{eq:123'}), we conclude that
$$e^{(m+n)R(t)}=\kappa^{-1} h_{l,\beta}\left(e^{t-nR(t)}\right)\geq \kappa^{-1}h_{l,\beta}\left(e^{t/2}\right)=2^{-l}\kappa^{-1}\max\{t,\beta\}^{l-1}\log\max\{t,\beta\}.$$
\end{proof}

{\begin{cor}
\label{cor:RvsRstar}
Let $\psi$ be as in Lemma \ref{lem:ridof1},  and let $R$ and $R^*$ be the functions corresponding to $\psi$ through Lemma \ref{lem:C1} for $(m,n)=(l,1)$ and $(m,n)=(1,l)$ respectively. Then, if $|\log \kappa|\leq e^{\beta}$, it holds that
$$|R(t)-R^*(t)|=O_{l}(\beta)$$
 for all $t\geq 0$.
\end{cor}}
{
\begin{proof}
If $t\geq 4\max\{|\log\kappa|,l\log\beta\}$, by (\ref{eq:111}) and (\ref{eq:222}), we have that
$$2^{-l}\leq e^{(l+1)(R(t)-R^*(t))}\leq 2^l,$$
hence, we may assume that $t\leq 4\max\{|\log\kappa|,l\log\beta\}$. Then, again by (\ref{eq:111}), we have that
$$(l+1)|R(t)-R^*(t)|\leq l\log\left(4 \max\{|\log\kappa|,l\log \beta\}\right).$$
By the assumption that $|\log\kappa|\leq e^{\beta}$, we conclude.
\end{proof}}

The next technical lemma will be useful in Section \ref{sec:countingonav}. 

\begin{lem}
\label{lem:deriv}
Let $\psi,R,$ and $R^{*}$ be as in Corollary \ref{cor:RvsRstar}. Assume that $|\log \kappa|\leq e^\beta$. Then for all $a,b,c>0$ and $t\geq 0$ it holds that
$$\big|R^{*}(t+aR^{*}(t)+bR(t)+c)-R^{*}(t)\big|=O_{l}\left(a+\beta(1+b)+c\right).$$
\end{lem}

\begin{proof}
For $t\geq 0$ define $x(t):=e^{t-nR^{*}(t)}$. Recall that {the function $R^{*}(t)$ is continuous for all $t\geq 0$ and, by Remark \ref{rmk:4.2},} differentiable except perhaps at the unique value of $t$ for which $x(t)=e^{\beta}$. By {differentiating} (\ref{eq:123}) we find that
$$(l+1)(R^{*})'(t)e^{(l+1)R(t)}=\kappa^{-1}\left(\frac{d}{dx}h_{l,\beta}\right)(x)\left(1-(R^{*})'(t)\right).$$
Since $0\leq (R^{*})'(t)\leq 1$ (see Lemmas \ref{lem:C1} and \ref{lem:additional}), by using (\ref{eq:123}) once again, we deduce that 
$$0\leq (R^{*})'(t)\leq \frac{d h_{l,\beta}}{dx}(x)\frac{1}{h_{l,\beta}(x)}\leq\begin{cases}
0 &\mbox{if }x(t)< e^{\beta} \\
2x(t)^{-1} &\mbox{if }x(t)>e^{\beta}
\end{cases}.$$
It follows that for $x(t)<e^{\beta}$ one has that $R^{*}(t)(R^{*})'(t)=0$, while for $x(t)>e^{\beta}$
\begin{equation}
\label{eq:RR'}
R^{*}(t)(R^{*})'(t)\leq \frac{\log\left(\kappa^{-1}h_{l,\beta}(x)\right)}{x}\leq \frac{|\log\kappa|}{e^{\beta}}+\frac{\log h_{l,\beta}(x)}{x}=O_{l}(1).    
\end{equation}

Now, if $x\left(t+aR^{*}(t)+bR(t)+c\right)\leq e^{\beta}$, it must also be $x(t)\leq e^{\beta}$ (recall that both the functions $R$ and $R^{*}$ are $\geq 0$ and $t\mapsto t-nR(t)$ is strictly increasing). Hence, in this case, the thesis is obvious, since $R^{*}$ is constant. Thus, we may assume that $x\left(t+aR^{*}(t)+bR(t)+c\right)>e^{\beta}$. Let us also assume that $x(t)\geq e^{\beta}$. Then, by the Mean Value Theorem, there exists $\theta$ between $t$ and $t+aR^{*}(t)+bR(t)+c$ such that
$$R^*\left(t+aR(t)+bR^*(t)+c\right)= R^*(t)+(R^*)'(\theta)\left(aR^{*}(t)+bR(t)+c\right).$$
By Lemma \ref{cor:RvsRstar}, we deduce that
\begin{multline*}
R^*\left(t+aR(t)+bR^*(t)+c\right)\\
= R^*(t)+(R^*)'(\theta)R^{*}(\theta)\cdot \frac{(a+b)R^{*}(t)}{R^{*}(\theta)}+(R^{*})'(\theta)(O_l(b\beta)+c).
\end{multline*}
Then, from \eqref{eq:RR'} and the fact that $0\leq (R^{*})'(\theta)\leq 1$, we conclude that
$$R^*\left(t+aR(t)+bR^*(t)+c\right)=R^*(t)+O_{l}(a+b\beta+c) 
=R(t)+O_l(a+(1+b)\beta+c).$$
To finish the proof, we need to deal with the case $x(t)\leq e^{\beta}$. Define $t_{\beta}$ to be the unique $t$ such that $x(t_{\beta})=e^{\beta}$. Then it suffices to apply the Mean Value Theorem for the function $R^{*}$ on the interval $[t_{\beta},t+aR(t)+bR^*(t)+c]$ as above and observe that $R^{*}(t_{\beta})=R^{*}(t)$.
\end{proof}

\section{Dangerous Intervals}
\label{sec:dansets}

{Recall that {to prove Proposition \ref{prop:fibrestatement}
we fix $d\ge 2$, $ l = 1,\dots, d$, $S\subset\{2,\dotsc,d\}$ with $\#S=l-1$,  
$(\alpha_{2},\dotsc,\alpha_{d})\in[0,1]^{d-1}$ and    ${\gamma}>0$   such that \eqref{eq:multinductivesim} and \eqref{eq:multinductivedual} hold, and let $\bs{f}$ be as in \equ{deff}.}
{In this section}, based on Corollary \ref{cor:C3}, {for a positive constant $\kappa$ to be determined later,} we introduce a collection of "dangerous" sets in {an interval $I\subset \R$}, forming the complement of the set \eqref{mainset}
appearing in Proposition \ref{prop:fibrestatement}. 

We assume to work in some fixed interval $I_{0}=[0,L]\subset\mb{R}$ of length $L$ and with a fixed stretching parameter $\beta\geq e$. The selection of these parameters {together with $\kappa$} will be the object of Section \ref{sec:conclusion}. {Henceforth, we will denote by $R_{l}$ and by $R_{l}^{*}$ the functions corresponding to $\psi(x)=\min\{1,\kappa \psi_{l,\beta}(x)\}$ through Lemma \ref{lem:C1} in the cases $(m,n)=(l,1)$ and $(m,n)=(1,l)$ respectively. Note that $\psi$ has the properties required to apply Corollary \ref{cor:C3}.

\begin{definition}
\label{def:dualdan}
For $1\leq l\leq d$, $S\subset \{2,\dotsc,d\}$ with $\#S= l-1$, $\bs{b}\in\mb{Z}^{l}\setminus\{\bs{0}\}$, $b_{0}\in\mb{Z}$, and $\bs{t}\in \beta\mb{Z}^{l}$  with ${t=\sum_{i}t_{i}\ge 0}$  and $t_{i}\geq R_{l}^*(t)$, we define the $(S,\bs{t},b_{0},\bs{b})$-"dual dangerous {interval}" as
{$$D_{\bs{t}}^{*}(S,b_{0},\bs{b}):=\left\{x\in I_0 : \left\|a(t,\bs{t}){u_{(x,{\pi_{S}\bs{\alpha}^{\scriptscriptstyle{T}})}} {b_{0}\choose \bs{b}}}\right\|<e^{-R_l^*(t)}\right\}.$$
One can see, using \equ{deff}, \equ{lambday}, \equ{1min} and \equ{defatu}, that equivalently}
\begin{equation}
\label{eq:Ddual}
D_{\bs{t}}^{*}(S,b_{0},\bs{b})=\left\{x\in I_{0}: \begin{aligned}
 & |{f_{b_{0},\bs{b}}}(x)|< e^{-t-R_{l}^*(t)} \\
 & |b_{i}|< e^{t_{i}-R_{l}^*(t)}\mbox{ for } i\in \{1\}\cup S\\
\end{aligned}\right\},
\end{equation}
where
\begin{equation}
\label{eq:defff}
{f_{b_{0},\bs{b}}(x):=b_{0}+\bs{b}\cdot\pi_{\{1\}\cup S}\bs{f}(x) = b_{0}+b_1x + \sum_{i\in S}b_i\alpha_i.}
\end{equation}
\end{definition}
{It is clear from \eqref{eq:Ddual} and \eqref{eq:defff}  that $D_{\bs{t}}^{*}(S,b_{0},\bs{b})$ is a subinterval of $I_0$.} 

\begin{definition}
\label{def:simdan}
For $1\leq l\leq d$, $S\subset \{2,\dotsc,d\}$ with $\#S= l-1$, $\bs{b}\in\mb{Z}^{l}$, $b_{0}\in\mb{Z}\setminus\{0\}$, and $\bs{t}\in \beta\mb{Z}^{l}$  with ${t=\sum_{i}t_{i}\ge 0}$ and $t_{i}\geq -R_{{l}}(t)$, we define the $(S,\bs{t},b_{0},\bs{b})$-"simultaneous dangerous {interval}" as
{$$D_{\bs{t}}(S,b_{0},\bs{b}):=\left\{x\in I_0 : \left\|a(\bs{t},t){u_{x\choose {\pi_{S}\bs{\alpha}}} {  {\bs{b}\choose b_{0}}}}\right\|<e^{-R_l(t)}\right\},$$
or, equivalently,}
\begin{equation}
{D_{\bs{t}}(S,b_{0},\bs{b})=\left\{x\in I_{0}: \begin{aligned}
 & |b_{i}+b_{0}f_{i}(x)|<e^{-t_{i}-R_{l}(t)}\mbox{ for } i\in \{1\}\cup S \\
 & |b_{0}|<e^{t-R_{l}(t)} \\
\end{aligned}\right\}.\nonumber}
\end{equation}
\end{definition}

{Here and hereafter $f_i$ 
{denotes} the $i$-th component of the function $\bs{f}$, so that $f_1(x)=x$ and $f_i(x)=\alpha_i$ for $i>1$. {{Thus} it is {again} clear   that $D_{\bs{t}}(S,b_{0},\bs{b})$ is an interval.} 
}

{\begin{rmk} It is easy to see that the union
$$\bigcup_{(b_{0},\bs{b})\in\mb{Z}\times(\mb{Z}^{l}\setminus\{\bs{0}\})}D_{\bs{t}}^{*}(S,b_{0},\bs{b})$$
of all the dual dangerous {interval}s for fixed $S$ and $\bs{t}$ as above coincides with the set
$$\bigcup_{(b_{0},\bs{b})\in\mb{Z}^{l+1}\setminus\{\bs{0}\}}D_{\bs{t}}^{*}(S,b_{0},\bs{b}) = \left\{x\in I_{0}:  \delta\left(a(t,\bs{t})u_{(x,{\pi_{S}\bs{\alpha}^{\scriptscriptstyle{T}})}} \mb{Z}^{l+1}\right)<e^{-R_l^*(t)}\right\}.$$ 
Indeed, if $x\in D_{\bs{t}}^{*}(S,b_{0},\bs{0})$, in view of \eqref{eq:defff} we must have $|b_0| = |{f_{b_{0},\bs{0}}}(x)|< e^{-t-R_{l}^*(t)}\le 1$, hence $b_0 = 0$. Likewise, if $x\in D_{\bs{t}}(S,0,\bs{b})$, then for each $i$ we have that $$ |b_{i}+b_{0}f_{i}(x)| = |b_{i}|<e^{-t_{i}-R_{l}(t)}\le 1,$$ hence $\bs{b} = \bs{0}$. This implies that the union
$$\bigcup_{(b_{0},\bs{b})\in(\mb{Z}\setminus\{0\})\times\mb{Z}^{l}}D_{\bs{t}}(S,b_{0},\bs{b})$$
of all the simultaneous dangerous {interval}s for fixed $S$ and $\bs{t}$ as above coincides with the set
$$\bigcup_{(b_{0},\bs{b})\in\mb{Z}^{l+1}\setminus\{\bs{0}\}}D_{\bs{t}}(S,b_{0},\bs{b}) = \left\{x\in I_{0}:  \delta\left(a(\bs{t},t){u_{x\choose {\pi_{S}\bs{\alpha}}}}\mb{Z}^{l+1}\right)<e^{-R_l(t)}\right\}.$$ \end{rmk}
In light of the above remark and Corollary \ref{cor:C3},
the proof of Proposition \ref{prop:fibrestatement} reduces to showing the following}

\begin{prop}
\label{prop:complementCR}
There exist an interval $I_{0}$, constants $\kappa>0$, $\beta\geq 1$, and a sequence of {natural} numbers $r_{k}
$   with $r_{k}\to \infty $ such that for any fixed set of indices $
S \subset\{2,\dotsc,d\}$, with $\#S=l-1$ ($1\leq l\leq d$), the complement of the union
\begin{equation}
\label{eq:target}
\bigcup_{\substack{\bs{t}\in \beta\mb{Z}^{l}\\ {t\ge 0} \\ t_{i}\geq -R_{l}(t) }}\bigcup_{(b_{0},\bs{b}){\in(\mb{Z}\setminus\{0\})\times\mb{Z}^{l}}}D_{\bs{t}}(S,b_{0},\bs{b})\cup \bigcup_{\substack{\bs{t}\in \beta\mb{Z}^{l} \\ {t\ge 0}\\ t_{i}\geq R^{{*}}_{l}(t)}}\bigcup_{(b_{0},\bs{b}){\in\mb{Z}\times(\mb{Z}^{l}\setminus\{\bs{0}\})}}D_{\bs{t}}^{*}(S,b_{0},\bs{b})
\end{equation}
in $I_{0}$ is an $r_{k}$-Cantor-rich set. 
\end{prop}

{Its proof will occupy the rest of the paper. The next two lemmas highlight some useful properties of dual and simultaneous dangerous {interval}s.}

\begin{lem}
\label{lem:danlengthdual}
{If $\kappa\leq \gamma$, where $\gamma$ is the constant appearing in \eqref{eq:multinductivedual}, one of the following two cases occurs:
\begin{itemize}
    \item[$i)$] 
    $D_{\bs{t}}^{*}(S,b_{0},\bs{b})$ is contained in an interval of length $e^{-(t+t_{1})+\beta}$;\vspace{2mm}
    \item[$ii)$] 
    $D_{\bs{t}}^{*}(S,b_{0},\bs{b})$ is contained in {another} dangerous {interval} $D_{\bs{t}'}^{*}(S,b_{0},\bs{b})$, defined by the same integer vector $\binom{b_0}{\bs b}$ with parameters $t_{1}'< t_1$ and $t'< t$.
\end{itemize}}
\end{lem}

\begin{proof}
{Let us show that, when $\kappa\leq \gamma$ and $ii)$ does not occur, we have that
$$|b_{1}|\geq e^{t_{1}-\beta-R_{l}^*(t-\beta)}\geq e^{t_1-\beta-R_{l}^*(t)}.$$
The conclusion will follow by dividing by $|b_{1}|$ both sides of the first inequality in (\ref{eq:Ddual}). If $|b_{1}|< e^{t_{1}-\beta-R_{l}^*(t-\beta)}$, {then} $D_{\bs{t}}^{*}(S,b_{0},\bs{b})$ is contained in a larger dangerous {interval}, defined by the same inequalities but where $t$ and $t_{1}$ are replaced with $t-\beta$ and $t_{1}-\beta$ respectively (we use the fact that the function $R_{l}^*$ is non-decreasing). {This implies} $ii)$, contrary to our assumption. The only possible obstruction to this argument is represented by the case when $t_{1}< R_{l}^*(t-\beta)+\beta$. If $t_{1}< R_{l}^*(t-\beta)+\beta$, however, the fact that $|b_{1}|< e^{t_{1}-\beta-R_{l}^*(t-\beta)}$ implies that $|b_{1}|< e^{0}$ and, hence, $b_{1}=0$. By multiplying all inequalities in the definition of $D_{\bs{t}}^{*}(S,b_{0},\bs{b})$ and using (\ref{eq:123}), we find that any $x\in D_{\bs{t}}^{*}(S,b_{0},\bs{b})$ has the property that
$${|b_{2}|_{+}}\dotsm {|b_{l}|_{+}}|{f_{b_{0},\bs{b}}}(x)|<e^{-(l+1)R_{l}^*(t)}=\kappa h_{l}\left({|b_{2}|_{+}}\dotsm {|b_{l}|_{+}}\right)^{-1}\leq \kappa h_{l-1}\left({|b_{2}|_{+}}\dotsm {|b_{l}|_{+}}\right)^{-1}.$$
However, if $\kappa\leq \gamma$, this contradicts (\ref{eq:multinductivedual}), showing that the {intervals}  for which \linebreak   $|b_{1}|< e^{t_{1}-\beta-R_{l}^*(t-\beta)}$ and $t_{1}< R_{l}^*(t-\beta)+\beta$ are empty.} 
\end{proof}


\begin{lem}
\label{lem:danlength}
{If $\kappa\leq e^{-(l+1)\beta}$, one of the following two cases occurs:
\begin{itemize}
    \item[$i)$] 
    $D_{\bs{t}}(S,b_{0},\bs{b})$ is contained in an interval of length $e^{-(t+t_{1})+\beta}.$\vspace{2mm} 
    \item[$ii)$] 
    $D_{\bs{t}}(S,b_{0},\bs{b})$ is contained in {another} dangerous {interval} $D_{\bs{t}'}(S,b_{0},\bs{b})$, defined by the same integer vector $\binom{b_0}{\bs b}$ with parameters $t_1'\leq t_1$ and $t'< t$. 
\end{itemize}
}
\end{lem}

\begin{proof}
{
{Similarly to the proof of Lemma \ref{lem:danlengthdual}}, if $ii)$ does not occur, we may assume that $|b_{0}|\geq e^{t-\beta-R_{l}(t-\beta)}$, since otherwise the set $D_{\bs{t}}(S,b_{0},\bs{b})$ is contained in a larger dangerous {interval}, obtained by decreasing any of the parameters $t_{i}$. The only possible exception to this argument is represented by the case when $t_{i}< -R_{l}(t-\beta)+\beta$ for all $i$. 

If $t_{i}<-R_{l}(t-\beta)+\beta$ for all $i$, it {follows that} $t\leq -(l+1)R_{l}(t)+\beta(l+1)$. Since  $\kappa\leq e^{-(l+1)\beta}$ and $h_{l,\beta}\geq 1$, (\ref{eq:123}) implies that $R_{l}(0)\geq |\log\kappa|\geq \beta$, whence $$t\leq -(l+1)R_{l}(t)+\beta(l+1)\leq -(l+1)R_{l}(0)+\beta(l+1)\leq 0,$$ a contradiction to $|b_0|<e^{t-R_l(t)}$. Then it follows from the inequality
$$|b_{1}+b_{0}f_{1}(x)|<e^{-t_{1}-R_{l}(t)}$$
that 
${D_{\bs{t}}
(S,b_{0},\bs{b})}$ is an interval of length $e^{-(t+t_{1})+\beta}$.}
\end{proof}

We now prove two additional lemmas, ensuring that all points {$y$} outside dangerous intervals $D_{\bs{t}}^{*}(S,b_{0},\bs{b})$ and $D_{\bs{t}}(S,b_{0},\bs{b})$ for all times $\bs{t}$ with bounded sum $t$ have large products ${|b_{1}|_{+}}\dotsm {|b_{l}|_{+}}|{f_{b_{0},\bs{b}}}(y)|$ and $b_{0}|b_{1}+b_{0}f_{1}(y)|\dotsm|b_{l}+b_{0}f_{l}(y)|$ {respectively}, provided the components of the vector $\bs{b}$ are small enough. The proof is similar to that of Corollary \ref{cor:C3}, but since we require a slightly more precise result, taking into account bounds on the time $\bs{t}$, we report the details. Note that the dual and simultaneous versions are slightly different. The reasons for this will become clear in Section \ref{sec:countingonav}.

\begin{lem}
\label{lem:dualinductivehyp}
Let $S\subset \{{2},\dotsc, d\}$ with $\# S=l-1$ and $l\geq 1$. Let $T\in\mb{N}$ and $y\in I_{0}$ such that $y\notin D_{\bs{t}}^{*}(S,b_{0},\bs{b})$ for any $(b_{0},\bs{b})\in\mb Z^{l+1}$ and any $\bs{t}\in\beta\mb{Z}^{l}$, with $t_{i}>R_{l}^{*}(t)$ and $t+t_{1}<T$. Then  for all $\bs{s}\in\beta\mb{Z}^{l}$ with $s_{i}>R_{l}^{*}(s)$ and $s+s_{1}<T$, and all $(b_{0},\bs{b})$ such that ${|b_{i}|_{+}}<e^{s_{i}-R_{l}^{*}(s)-2l\beta }$, we have that
$${|b_{1}|_{+}}\dotsm {|b_{l}|_{+}}|{f_{b_{0},\bs{b}}}(y)|\geq e^{-(l+1)R_{l}^{*}(s)+O_{l}(\beta)},$$
where $f_{(b_0,\bs b)}$ is as in \eqref{eq:defff}.
\end{lem}

\begin{proof}
{Pick $(b_0,\bs b)$ as above and let $\bs{s}'\in\mb{R}^{l}$ such that ${|b_{i}|_{+}}=e^{s_{i}'-R_{l}^{*}(s')-2l\beta }$ for $i=1,\dotsc,l$. Finding such numbers $s_{i}'$ is always possible, since the map $s\mapsto s-lR_{l}^{*}(s)$ is strictly increasing. Then it must be $s_{i}'>R_{l}^{*}(s')$ for all $i$ and $s'<s$ (by the hypothesis). Moreover, one has that 
$$e^{s_{1}'-R_{l}^{*}(s')}<e^{s_{1}-R_{l}^{*}(s)},$$
and since $R_{l}^{*}$ is non-decreasing, it must be $s_{1}'\leq s_{1}$. Hence, $s'+s_{1}'<s+s_{1}<T$.
Assume by contradiction that
$${|b_{1}|_{+}}\dotsm {|b_{l}|_{+}}{|{f_{b_{0},\bs{b}}}(y)|}< e^{-(l+1)(R_{l}^{*}(s')+2l\beta )}.$$
Then, for some $\bs{s}''$ very close to $\bs{s}'$ (with $s_{i}''>s_{i}'$) and $1\leq A=e^{s''-l(R_{l}^{*}(s'')+2l\beta)}$ {we must have} $\pi_{\{1\}\cup S}\bs{f}(y)\in\mc{S}_{1,l}^{\times}\left(\tilde{\psi},A\right)$, with $\tilde{\psi}$ corresponding to the function $\tilde R^*:=R_{l}^{*}+2\beta l$ through Lemma \ref{lem:C1}.
On the other hand, with the notation of Proposition \ref{prop:C2}, the hypothesis implies that
\begin{equation}
\label{eq:newmin}
\delta\left(a(t,\bs{t})\Lambda_{\pi_{\{1\}\cup S}\bs{f}(y)^{\scriptscriptstyle{T}}}\right)> e^{-R_{l}^{*}(t)-2l\beta }=e^{-\tilde R^*(t)}
\end{equation}
for all $\bs{t}\in\mb{R}^{l}$, with $t_{i}>R_{l}^{*}(t)$ and $t+t_{1}<T$. This follows from an argument similar to that used in Corollary \ref{cor:C3}.
In particular, one can find $\bs t'\in\beta\mb Z^l$, with $t_i'>R_{l}^{*}(t')$ and $t'+t_1'<T$, such that
$\|\bs t-\bs t'\|\leq \beta l$, deducing (\ref{eq:newmin}) (recall {that ${\frac{\dd}{\dd t}}R_{l}^{*}<1$}).
This can be done, for example, by increasing all the components $t_i$ for $i\neq i_0$ ($i_0$ being a fixed index) by a quantity between $0$ and $\beta$ and by decreasing $t_{i_0}$ by a quantity between $0$ and $l\beta $. If this cannot be done, then all the components $t_i$ are smaller than $R_{l}^{*}(t)+l\beta $, showing that the minimum in (\ref{eq:newmin}) must be at least $e^{-R_i(t)-l\beta}$.
{Thus}, we deduce from Proposition \ref{prop:C2} that
$$\pi_{\{1\}\cup S}\bs{f}(y)^{\scriptscriptstyle{T}}\notin\mc{S}_{1,l}^{\times}\left(\tilde{\psi},e^{t-l\tilde R^*(t)}\right)$$
for all $t$ coming from a $\bs{t}\in\mb{R}^{l}$ with $t_{i}>R_{l}^{*}(t)$, $t+t_{1}<T$. By taking $t=s''$ we find a contradiction.}
\end{proof}

\begin{lem}
\label{lem:siminductivehyp}
Let $S\subset\{2,\dotsc,d\}$ with $\# S=l-1$ and $l\geq 1$. Let $T\in\mb{N}$ and $y\in I_{0}$ such that $y\notin D_{\bs{t}}(S,b_{0},\bs{b})$ for any $(b_{0},\bs{b})\in\mb Z^{l+1}$ and any $\bs{t}\in\beta\mb{Z}^{l}$, with $t_{i}>-R_{l}(t)$ and $0\leq t<T$. Then for all $s\in\beta\mb{Z}$ with $0\leq s<T$, and all $(b_{0},\bs{b})$ such that $|b_{0}|<e^{s-R_{l}(s)-2l^2\beta }$, we have that
$$|b_{0}|\prod_{i\in \{1\}\cup S}|b_{i}+b_{0}f_{i}(y)|\geq e^{-(l+1)R_{l}(s)+O_{l}(\beta)}.$$
\end{lem}

\begin{proof}
{Without loss of generality, we may assume that $S=\{2,\dotsc,l\}$, that $|b_{i}+b_{0}f_{i}(y)|$ {is less than $1$} for all $i\in \{1\}\cup S$, and that $e^{s-(2l^2+1)\beta-R_{l}(s)}\leq |b_{0}|<e^{s-2l^2\beta -R_{l}(s)}$, since $R_l$ is non-increasing. Let us first consider the case when $y\neq -b_{1}/b_{0}$. Since $0<|b_{i}+b_{0}f_{i}(y)|<1$ for all $i$, we can find $\bs{s}'\in\mb{R}^{l}$, with $s_{i}'>-R_{l}(s')$, such that $|b_{i}+b_{0}f_{i}(y)|=e^{-s_{i}'-R_{l}(s')}$ for $i=1,\dotsc,l$. Finding the numbers $s_{i}'$ is always possible, since the map $s\mapsto s+lR_{l}(s)$ is strictly increasing. Let $\tilde{\psi}$ be the function corresponding to $\tilde{R}:=R_l+2l^2\beta $ through Lemma \ref{lem:C1}. We aim to show that $s'\leq s+2l^{3}\beta $. Assume not, then we have that
\begin{equation}
\label{eq:prod}
\prod_{i=1}^{l}|b_{i}+b_{0}f_{i}(y)|=e^{-s'-lR_{l}(s')}< e^{-s-l(R_{l}(s)+2l^2\beta )}=\tilde{\psi}\left(e^{s-\tilde{R}(s)}\right).
\end{equation}
Hence, given that $|b_{0}|<e^{s-R_{l}(s)-2l^2\beta}$, we find that
$$\pi_{\{1\}\cup S}\bs{f}(y)\in \mc{S}_{l,1}^{\times}\left(\tilde{\psi},e^{s-\tilde{R}(s)}\right).$$
On the other hand, with the notation of Proposition \ref{prop:C2}, the hypothesis implies that
\begin{equation}
\label{eq:newmin1}
\delta\left(a(\bs{t},t)\Lambda_{\pi_{\{1\}\cup S}\bs{f}(y)}\right)> e^{-R_{l}(t)-2l^2\beta }=e^{-\tilde{R}(t)}
\end{equation}
for all $\bs{t}\in\mb{R}^{l}$ with $t_{i}>-R_{l}(t)$ and $t<T$. To see this, we use, once again, the argument appearing in the proof of Corollary \ref{cor:C3}. In particular, it suffices to build a new vector $\bs t'\in\beta\mb Z^l$ by rounding up all the negative components of $\bs t$ by a quantity between $\beta$ and $2\beta$ and all of its positive components by a quantity between $0$ and $\beta$, except one, denoted by $t_{i_0}$, which we decrease (or increase) by at most $2(l-1)\beta$ in order to ensure $t'<t$ and $|t-t'|<\beta$. If $t'<0$, then (\ref{eq:newmin1}) is obvious, since $|t-t'|<\beta$ implies $t<\beta$. Conversely, if $t'>0$, (\ref{eq:newmin1}) follows from the hypothesis applied to $\bs t'$. The only case when this argument fails is when for all $i$ it holds that $t_i-2(l-1)\beta<-R_i(t)$ and decreasing $t_{i_0}$ is not allowed. However, in this case, we have that $t\leq 2l(l-1)\beta$, and (\ref{eq:newmin1}) is once again trivially true. Therefore it follows from Proposition \ref{prop:C2} that
\begin{equation}
\label{eq:prod1}
\pi_{\{1\}\cup S}\bs{f}(y)\notin\mc{S}_{l,1}^{\times}\left(\tilde{\psi},e^{t-\tilde{R}(t)}\right)
\end{equation}
for all $t$ coming from a $\bs{t}\in\mb{R}^{l}$ with $t_{i}>-R_{l}(t)$  and $0\leq t<T$, {where $\tilde{\psi}$ corresponds} to the function $R_{l}+2l^2\beta $ through Lemma \ref{lem:C1}. By taking $t=s$ we find a contradiction. Then from $s'\leq s+2l^3\beta$ and $e^{s-(2l^2+1)\beta -R_{l}(s)}\leq |b_{0}|$ we deduce the claim. Finally, let us show that $y\neq -b_{0}/b_{1}$. If this were the case, we would have
$$\pi_{\{1\}\cup S}\bs{f}(y)\in \mc{S}_{l,1}^{\times}\left(\tilde{\psi},e^{s-\tilde{R}(s)}\right),$$
as (\ref{eq:prod}) is trivially satisfied. This, however, contradicts (\ref{eq:prod1}), showing that this case never occurs.} 
\end{proof}

\section{Setup of the Removing Procedure}
\label{sec:setup}

{From now on, we will fix $l$ and $S$ in Definitions \ref{def:dualdan} and \ref{def:simdan} (the same in both cases). {Here and hereafter, $R_l$ and $R_l^*$ will denote the functions corresponding to $\min\{1,\kappa\psi_{l,\beta}\}$ through Lemma \ref{lem:C1} for $(m,n)=(l,1)$ and $(m,n)=(1,l)$ respectively, with $\beta\geq e$, $0<\kappa<\gamma$, and $\gamma$ as in (\ref{eq:multinductivesim}) and (\ref{eq:multinductivedual}). In Section \ref{sec:conclusion}, further assumptions on the constants $\beta$ and $\kappa$ will be made. To simplify the notation, we will also drop the subscript $l$ in the functions $R_{l}$ and $R_l^*$. At times, we will need to use the functions $R_{s}$ and $R_s^*$ associated to the map $\min\{1,\kappa \psi_{s,\beta}\}$ for some index $s\leq l$. In this case, the index $s$ will always be indicated. It is worthwhile to note that $R_{s-1}\geq R_{s}$ and $R_{s}^*\geq R_{s-1}^{*}$ for all $s=2,\dotsc,l$, as a consequence of (\ref{eq:123}). The symbol $I_{0}$ will denote the interval $[0,L]$ of length $L$ yet to be chosen. This quantity will again be discussed in Section \ref{sec:conclusion}.}}

{Note that, when $l=1$, dual and simultaneous dangerous intervals coincide, i.e.\ one has that
$$D_{\bs{t}}^{*}(S,b_{0},b_{1})=D_{\bs{t}}(S,b_{1},b_{0})\vspace{2mm}$$
for all multi-times $\bs{t}$ and all vectors $(b_{0},b_{1})$. In view of this, we will deal with the case $l=1$ only in the dual setting (the proof is analogous but no inductive hypothesis is required). One should also notice that the case $d=1$ coincides with the case $l=1$ for each $d>1$. The proof in these cases is essentially the same, but requires no inductive hypothesis. We will comment on the case $l=1$ further, in the proof of Lemma \ref{lem:alblocksdual}, where the only difference compared to the cases $l>1$ occurs.}

To prove Proposition \ref{prop:complementCR}, we aim to construct an $r_{k}$-Cantor-type set {with $\mc{I}_0  = \{I_{0}\}$} that avoids both types of dangerous intervals. 
Let $F:\mb{N}\cup\{0\}\to\mb{R}$ be the function $F(k):=\prod_{i={0}}^{k}r_{i}$, where $r_{k}
$ is a sequence of {natural} numbers   such that $r_{k}\to \infty$. By definition, the $k$-th level intervals of an $r_{k}$-Cantor-type set have length
$$|I_{0}|\prod_{i=0}^{k-1}r_{i}^{-1}=\frac{L}{F(k-1)}.$$
  {Then for any $k\in\N$} from the collection $\mc{I}_{k-1}/r_{k-1}$ we remove all the intervals $I_{k}$ that are intersected by the sets $\bigcup_{(b_{0},\bs{b})}D_{\bs{t}}(S,b_{0},\bs{b})$ and $\bigcup_{(b_{0},\bs{b})}D_{\bs{t}}^{*}(S,b_{0},\bs{b})$ for which $T=t+t_{1}$ satisfies
\begin{equation}
L^{-1}F(k-2)\leq e^{T}<L^{-1}F(k-1).\nonumber
\end{equation}
With this definition, we have that $\mc{K}(\mc{I})$ lies in the complement of (\ref{eq:target}). Note that we can always assume $T\geq 0$ and hence that $k\geq 2$, as $t-nR^{(*)}(t)\geq 0$ (see Corollary \ref{cor:C3}).

By Lemmas \ref{lem:danlengthdual} and \ref{lem:danlength}, the dangerous intervals that we remove from the collection $\mc{I}_{k-1}/r_{k-1}$ have length
$$e^{-T+\beta}\geq \frac{L}{F(k-1)}.$$
Hence, if $\hat{\mc{I}}_{k}$ denotes the collection of intervals removed at the $k$-th step, we have
\begin{multline*}
\#\hat{\mc{I}}_{k}\leq\sum_{\bs{t}\in\mc{D}(k)}\#\left\{(b_{0},\bs{b})\in\mb{Z}^{l+1}:D_{\bs{t}}^{*}(S,b_{0},\bs{b})\neq\varnothing\right\}\cdot\frac{e^{-T+\beta}}{|I_{k}|} \\
+\sum_{\bs{t}\in\mc{S}(k)}\#\left\{(b_{0},\bs{b})\in\mb{Z}^{l+1}:D_{\bs{t}}(S,b_{0},\bs{b})\neq\varnothing\right\}\cdot\frac{e^{-T+\beta}}{|I_{k}|},
\end{multline*}
where 
$$\mc{D}(k):=\left\{\bs{t}\in\mb{Z}^{l}:\ t_{i}\geq R^*(t)\mbox{ and }L^{-1}F(k-2)\leq e^{T}<L^{-1}F(k-1)\right\}$$
and
$$\mc{S}(k):=\left\{\bs{t}\in\mb{Z}^{l}:\ t_{i}\geq -R(t)\mbox{ and }L^{-1}F(k-2)\leq e^{T}<L^{-1}F(k-1)\right\}.$$

To apply Corollary \ref{cor:CRFullHaus}, however, we need to estimate the local characteristic $\Delta_{k}$ of the family $\mc{I}_{k}$. In order to do so, we partition the collection $\hat{\mc{I}}_{k}$ into $k$ sets and assign to each of these sets an index $p$ for $p\in\{0,\dotsc,k-1\}$. Then we estimate how many intervals in the family $\hat{\mc{I}}_{k}$ are contained in any interval $I_{p}\in\mc{I}_{p}$. It will be convenient to choose the partition of $\hat{\mc{I}}_{k}$ that assigns all intervals to only one fixed index $p=p(k)$, which will be defined later. With this choice, for each interval $I_{p}\in\mc{I}_{p}$ we are left to estimate
$$\#\hat{\mc{I}}_{k}\sqcap I_{p}:=\#\{I_{k}\in\hat{\mc{I}}_{k}:I_{k}\subset I_{p}\}.$$
From the discussion above, we deduce that
\begin{multline}
\#\hat{\mc{I}}_{k}\sqcap I_{p}\leq  \sum_{\bs{t}\in\mc{D}(k)}\#\left\{(b_{0},\bs{b})\in\mb{Z}^{l+1}:D_{\bs{t}}^{*}(S,b_{0},\bs{b})\cap I_{p}\neq \varnothing\right\}\cdot\frac{e^{-T+\beta}}{|I_{k}|} \\
+\sum_{\bs{t}\in\mc{S}(k)}\#\left\{(b_{0},\bs{b})\in\mb{Z}^{l+1}:D_{\bs{t}}(S,b_{0},\bs{b})\cap I_{p}\neq\varnothing\right\}\cdot\frac{e^{-T+\beta}}{|I_{k}|}.\label{eq:finalform}
\end{multline}
The goal of the next section will be to analyze the terms $$\#\left\{(b_{0},\bs{b})\in\mb{Z}^{l+1}:D_{\bs{t}}^{(*)}(S,b_{0},\bs{b})\subset I_{p}\right\}$$
for a given integer $p$.

\section{Counting on Average}
\label{sec:countingonav}

{Throughout this section, the assumptions made at the beginning of Section \ref{sec:setup} will be in place.} Let $J\subset I_{0}$ be an interval of length $c_{2}e^{-T+(l+1)R^{(*)}(t)}$, where $0<c_{2}<1$ and $R^{(*)}$ denotes either one of the functions $R$ or $R^*$. Note that $J$ is $c_{2}e^{(l+1)R^{(*)}(t)}$ times longer than any dangerous interval at time $T$. We aim to show that, for fixed $\bs{t}$, the interval $J$ is "on average" intersected by only one dangerous interval of each type. More precisely, if we consider sufficiently many intervals $J$ in a row to form a block, the number of dangerous intervals intersecting this block, will coincide with the total number of intervals $J$ stacked together to form the block. 

The following lemmas make the above argument more precise. In particular, they allow us to estimate how many intervals $J$ must be taken so that the "average" behaviour starts to appear. We will always assume that Corollary \ref{cor:CRFullHaus} is applicable, i.e., that $|\log\kappa|\leq e^\beta$.

\subsection{Dual Case}

\begin{lem}
\label{lem:alblocksdual}
Let $\bs{t}$ be fixed with $T=t+t_{1}$ and $t_{i}\geq R^{*}(t)$, and consider an interval $J\subset I_{0}$ of length $c_{2}e^{-T+(l+1)R^*(t)}$ with $c_{2}\geq e^{-(l+1)R^*(0)}$. Assume that there exists a dangerous interval $D_{\bs{t}}^{*}(S,b_{0},\bs{b})$ such that $J\cap D_{\bs{t}}^{*}(S,b_{0},\bs{b})$ contains at least one point that was never removed in previous steps, i.e., not lying in any dangerous interval of the form $D_{\bs{t}'}^{*}(S,b_{0}',\bs{b}')$ for any $\bs{t}'$ with $t'+t_{1}'<T$, nor in the intervals $D_{\bs{t}'}^{*}(S',b_{0}',\bs{b}')$ or $D_{\bs{t}'}(S',b_{0}',\bs{b}')$ for any $S'$ with $\#S'< l$ and any multi-time $\bs{t}'$. Let
\begin{equation}
\mc{P}_{\bs{t}}^{*}(J,m):=\left\{(b_{0},\bs{b})\in\mb{Z}^{l+1}:D_{\bs{t}}^{*}(S,b_{0},\bs{b})\cap\bigcup_{i=0}^{m}M_{i}\neq\varnothing\right\},\nonumber
\end{equation}
where $M_{0}=J$, $|M_{i}|=|J|$ for all {$i\in\N$}, and the intervals $M_{i}$ and $M_{i+1}$ only share upper and lower endpoints respectively. Let also
$$m_{\bs{t}}^{*}(J):=\max\{m:\#\mc{P}_{\bs{t}}^{*}(J,i)\geq i\mbox{ for all }i\leq m\}$$
and
$$B_{\bs{t}}^{*}(J):=\bigcup_{i=0}^{m_{\bs{t}}^{*}(J)}M_{i}.$$
Then, if the constant $c_{2}$ is small enough in terms of $l$, we have that
$$m_{\bs{t}}^{*}(J)\leq e^{(l+1)^{2}R^*(t)+O_{l}(\beta)}.$$
\end{lem}

\begin{proof}
First, we observe that $N=m_{\bs{t}}^{*}(J)$ is well defined, since $$\sup_{m}\#\mc{P}_{\bs{t}}^{*}(J,m)\ll_{l} \max\{|I_{0}|,1\}e^{t+\max\{t_{1},\dotsc,t_{l}\}-(l+1)R^{*}(t)}$$ (this follows from the fact that $|b_{i}|< e^{t_{i}-R^{*}(t)}$ and $|b_{0}|\leq l\max\{|I_{0}|,1\}e^{\max\{t_{i}\}-R^{*}(t)}+1$). Moreover, by the definition of $m_{\bs{t}}^{*}(J)$, we have that
$$\#\mc{P}_{\bs{t}}^{*}(J,N)=N.$$
Let $y\in J$ be a point in some dangerous interval that was never removed in previous steps. Then for any $(b_{0},\bs{b})\in\mc{P}_{\bs{t}}^{*}(J,N)$ and any $x\in D_{\bs{t}}^{*}(S,b_{0},\bs{b})$ we have that
\begin{multline}
\label{eq:oneside}
|{f_{b_{0},\bs{b}}}(y)|\leq |{f_{b_{0},\bs{b}}}(x)|+|f'_{(b_{0},\bs{b})}(x)||x-y| \\
\leq e^{-t-R^{*}(t)}+ e^{t_{1}-R^{*}(t)}Nc_{2}e^{-(t+t_{1})+(l+1)R^{*}(t)}\leq 2c_{2}Ne^{-t+lR^{*}(t)},
\end{multline}
where we used the fact that $c_{2}\geq e^{-(l+1)R^{*}(t)}$. Now, we consider two separate cases. We start by assuming that the vectors  
$(b_{0},\bs{b})\in\mc{P}_{\bs{t}}^{*}(J,N)$ do not lie on a proper linear subspace. Let
$$\Lambda (y):=\begin{pmatrix}
1 & \bs{f}(y)^{\scriptscriptstyle{T}} \\
\bs{0} & I_{l}
\end{pmatrix}\mb{Z}^{l+1}$$
and
$$\mc{B}(N):=\left[-2c_{2}Ne^{-t+lR^{*}(t)},2c_{2}Ne^{-t+lR^{*}(t)}\right]\times\prod_{i=1}^{l}\left[-e^{t_{i}-R^{*}(t)},e^{t_{i}-R^{*}(t)}\right].$$
Then (\ref{eq:oneside}), along with the definition of dangerous set, implies that
\begin{equation}
\label{eq:lattice}
N\leq \#\left(\Lambda (y)\cap\mc{B}(N)\right).
\end{equation}
Since we assumed that the vectors in $\mc{P}_{\bs{t}}^{*}(J,N)$ do not lie on a proper linear subspace, by Theorem \ref{thm:Bli}, we find that
\begin{equation}
N\leq \#\left(\Lambda (y)\cap\mc{B}(N)\right)\leq l+l!\frac{\textup{Vol}\,\mc{B}(N)}{\det\Lambda (y)}\leq l+4l!c_{2}N.\nonumber    
\end{equation}
If the constant $c_{2}$ is small enough, we deduce that $N\leq 2l$, proving the claim. Thus, we may assume from now on that the set $\mc{P}_{\bs{t}}^{*}(J,N)$ spans 
a proper linear subspace.

Let $D_{N}:=\textup{diag}\left(N^{-\frac l{l+1}},N^{\frac1{l+1}},\dotsc,N^{\frac1{l+1}}\right)$. To any integer vector $(b_{0},\bs{b})\in\mc{P}_{\bs{t}}^{*}(J,N)$, we uniquely associate the vector 
\begin{equation*}
\left(N^{-\frac l{l+1}}e^{t}{f_{b_{0},\bs{b}}}(y),N^{\frac1{l+1}}e^{-t_{1}}b_{1},\dotsc,N^{\frac1{l+1}}e^{-t_{l}}b_{l}\right)
\end{equation*}
lying in the lattice
$$\Lambda_{\bs{t},N}(y):=D_{N}a (t,\bs{t})\Lambda (y).$$
Then, by (\ref{eq:lattice}), we have that
$$N\leq \#\left(\Lambda_{\bs{t},N}(y)\cap D_{N}a (t,\bs{t})\mc{B}(N)\right).$$
Since the vectors $(b_{0},\bs{b})$ lie on a hyperplane, we deduce from Corollary \ref{cor:counting} that
\begin{equation}
\label{eq:counting}
N\ll_{l}1+\frac{c_{2}N^{\frac1{l+1}}e^{lR^{*}(t)}}{\lambda_{1}}+\dotsb+\frac{c_{2}N^{\frac l{l+1}}e^{R^{*}(t)}}{{\lambda}_{l}},
\end{equation}
where $\lambda_{i}$ denotes the first minimum of the lattice $\bigwedge^{i}\Lambda_{\bs{t},N}(y)$ for $i=1,\dotsc,l$. Here, we used the fact that $c_{2}e^{lR^{*}(t)}\geq e^{-R^{*}(t)}$. We will estimate $N$ by conducting a careful analysis of the quantities $\lambda_{1}$ and $\lambda_{1}^{*}$ (see Appendix \ref{sec:AppA} for the notation), where $\lambda_{1}^{*}$ denotes the first minimum of the lattice
$$\Lambda_{\bs{t},N}^{*}(y):=\left(\Lambda_{\bs{t},N}(y)^{-1}\right)^{\scriptscriptstyle{T}}.$$
We aim to show that one of the following three cases holds
\begin{itemize}
\item[$i)$] $\lambda_{1}\geq e^{-R^{*}(t)+O_{l}(\beta)}$ or $\lambda_{1}^{*}\geq e^{-R^{*}(t)+O_{l}(\beta)}$;\vspace{2mm}
\item[$ii)$]$\lambda_{1}\lambda_{1}^{*}\geq e^{-2R^{*}(t)+O_{l}(\beta)}$;\vspace{2mm}
\item[$iii)$]$\lambda_{1}^{*}\geq N^{-\frac1{l+1}}e^{R^{*}(t)}$.
\end{itemize}

Suppose that $\lambda_{1}$ is attained by the vector
$$\bs{v}_{1}:=\left(N^{-\frac l{l+1}}e^{t}{f_{b_{0},\bs{b}}}(y),N^{\frac1{l+1}}e^{-t_{1}}b_{1},\dotsc,N^{\frac1{l+1}}e^{-t_{l}}b_{l}\right).$$
If for some $i\in\{1,\dotsc,l\}$ it holds that $|b_{i}|\geq e^{t_{i}-2\beta l-R^{*}(t)}$, we deduce that
$$\lambda_{1}={\|\bs{v}_{1}\|}\geq N^{\frac1{l+1}}e^{-R^{*}(t)-2\beta l}\geq e^{-R^{*}(t)-2\beta l}.$$
Hence, $(i)$ occurs. In view of this, we can assume henceforth that for $i=1,\dotsc,l$ it holds that
\begin{equation}
\label{eq:smallai}
|b_{i}|\leq e^{t_{i}-2\beta l-R^{*}(t)}.
\end{equation}
From this point on, there are two main cases to consider. If all of the coefficients $b_{i}$ in $\bs{v}$ are non-zero, by (\ref{eq:AM-GM}) and Lemma \ref{lem:dualinductivehyp} applied to the set $S$, we deduce that
\begin{equation}
\lambda_{1}={\|\bs{v}_{1}\|}\geq \left(|{f_{b_{0},\bs{b}}}(y)||b_{1}|\dotsm|b_{l}|\right)^{\frac1{l+1}} 
=\left(|{f_{b_{0},\bs{b}}}(y)|{|b_{1}|_{+}}\dotsm {|b_{l}|_{+}}\right)^{\frac1{l+1}}\geq e^{-R^{*}(t)+O_{l}(\beta)},\nonumber
\end{equation}
where the last inequality follows form Lemma \ref{lem:dualinductivehyp} together with the assumption that the point $y$ was never removed. This shows $(i)$. If at least one of the coefficients $b_{i}$ is null, we have to do some more work. Here the minimum $\lambda_{1}^{*}$ will play a crucial role. Note that, if $l=1$, this case never occurs and we can directly assume that $(i)$ holds. This proves Lemma \ref{lem:alblocksdual} in the base case.

We start by considering possible values for the minimum $\lambda_{1}$. If $\bs{b}=\bs{0}$, we find that
\begin{equation}
\label{eq:case0}
\lambda_{1}=N^{-\frac l{l+1}}e^{t}.
\end{equation}
If $\bs{b}\neq \bs{0}$ and $b_{1}\neq 0$, we may assume without loss of generality that $b_{1},\dotsc,b_{s}\neq 0$ with $s<l$ and that the remaining components of the vector $\bs{b}$ are null. Then, we use (\ref{eq:AM-GM}) and Lemma \ref{lem:dualinductivehyp} applied to the set $S'=\{2,\dotsc,s\}$ and with parameter $T=\sum_{i=1}^{s}t_i+t_1$, to obtain
\begin{multline}
\label{eq:cases}
\lambda_{1}={\|\bs{v}_{1}\|}\geq\left(N^{-\frac {l-s}{l+1}}e^{t_{s+1}+\dotsb+t_{l}}|{f_{b_{0},\bs{b}}}(y)||b_{1}|\dotsm|b_{s}|\right)^{\frac 1{s+1}} \\
\geq N^{-\frac{l-s}{(l+1)(s+1)}}e^{\frac{t_{s+1}+\dotsb +t_{l}}{s+1}-R^{*}_{s}(t)+O_{l}(\beta)}.
\end{multline}
If $b_{0}\neq 0$ and $b_{1}=0$, we may assume that $b_{2},\dotsc,b_{s}\neq 0$ with $2\leq s\leq l$ and that the remaining components of the vector $\bs{b}$ are null. Then, by (\ref{eq:multinductivedual}) and (\ref{eq:123}), we find that
\begin{align}
& \lambda_{1}={\|\bs{v}_{1}\|}\geq\left(N^{-\frac {l-(s-1)}{l+1}}e^{t_{1}+t_{s+1}+\dotsb+t_{l}}|{f_{b_{0},\bs{b}}}(y)||b_{2}|\dotsm|b_{s}|\right)^{1/s}\nonumber \\
& \geq \left(N^{-\frac {l-(s-1)}{l+1}}e^{t_{1}+t_{s+1}+\dotsb+t_{l}}{\gamma}h_{s-1}\left(|b_{2}|\dotsm|b_{s}|\right)^{-1}\right)^{1/s}\nonumber \\
& \geq N^{-\frac {l-(s-1)}{(l+1)s}}e^{\frac{t_{1}+t_{s+1}+\dotsb +t_{l}}s-R^{*}_{s-1}(t)}\label{eq:casesa10},
\end{align}
{where we have used the fact that (\ref{eq:smallai}) and (\ref{eq:123}) imply that
$${\gamma}h_{s-1}\left(|b_{2}|\dotsm|b_{s}|\right)^{-1}\geq \kappa h_{s-1}\left(|b_{2}|\dotsm|b_{s}|\right)^{-1}\geq e^{-sR^{*}_{s-1}(t)}.$$
}

Now, we proceed to analyze the minimum $\lambda_{1}^{*}$. To this end, we choose a vector $\bs{v}_{1}^{*}$ whose length is equal to $\lambda_{1}^{*}$. Then $\bs{v}_{1}^{*}$ has the form
\begin{equation}
\label{eq:min1*}
\bs{v}_{1}^{*}=\left(N^{\frac l{l+1}}e^{-t}d_{0},N^{-\frac1{l+1}}e^{t_{1}}(d_{1}-d_{0}f_{1}(y)),\dotsc,N^{-\frac1{l+1}}e^{t_{l}}\left(d_{l}-d_{0}f_{l}(y)\right)\right),
\end{equation}
with $(d_{0},\bs{d})\in\mb{Z}^{l+1}$.
If $d_{0}\neq 0$ and (\ref{eq:case0}) holds, we conclude that
$$\lambda_{1}\lambda_{1}^{*}\geq \left(N^{-\frac l{l+1}}e^{t}\right)\left(N^{\frac l{l+1}}e^{-t}\right)=1,$$
whence $(ii)$. If $d_{0}\neq 0$ and either (\ref{eq:cases}) or (\ref{eq:casesa10}) occur, we have two further possible cases. Either $|d_{0}|\geq e^{t-R(t)-2l^2\beta }$ and hence $(i)$ follows from (\ref{eq:min1*}) and Corollary \ref{cor:RvsRstar}, or $|d_{0}|<e^{t-R(t)-2l^2\beta}$. If
\begin{equation}
\label{eq:smalld0}
|d_{0}|<e^{t-R(t)-2l^2\beta}
\end{equation}
and (\ref{eq:cases}) holds, we apply (\ref{eq:AM-GM}) to the first $s+1<l+1$ components of the vector $\bs{v}_{1}^{*}$ (where we assume $y$ irrational). By Lemma \ref{lem:siminductivehyp} applied to the set $S'=\{2,\dotsc,s\}$ with $T=t$, we obtain
\begin{align}
\label{eq:lastmin}
& \lambda_{1}^{*}\geq {\|\bs{v}_{1}^{*}\|} \\
& \geq\left(N^{\frac {l-s}{l+1}}e^{-t_{s+1}-\dotsb-t_{l}}|d_{0}|\left|-d_{1}+d_{0}f_{1}(y)\right|\dotsm\left|-d_{s}+d_{0}f_{s}(y)\right|\right)^{\frac 1{s+1}}\nonumber \\
& \geq N^{\frac{l-s}{(l+1)(s+1)}}e^{(-t_{s+1}-\dotsb -t_{l})/(s+1)-R^{*}_{s}(t)+O_{l}(\beta)},\nonumber 
\end{align}
where in the last inequality we used again Corollary \ref{cor:RvsRstar}, to compare $R$ with $R^*$.
If $|d_{0}|<e^{t-R^{*}(t)-2l^2\beta}$ and (\ref{eq:casesa10}) holds, we apply (\ref{eq:AM-GM}) to the components $0$ and $2,\dotsc,s$ of the vector $\bs{v}_{1}^{*}$. By (\ref{eq:multinductivesim}) and (\ref{eq:123}), we deduce that
\begin{align}
\label{eq:lastmina10}
& \lambda_{1}^{*}\geq {\|\bs{v}_{1}^{*}\|} \\
& \geq\left(N^{\frac {l-(s-1)}{l+1}}e^{-t_{1}-t_{s+1}-\dotsb-t_{l}}|d_{0}|\left|-d_{2}+d_{0}f_{1}(y)\right|\dotsm\left|-d_{s}+d_{0}f_{s}(y)\right|\right)^{1/s}\nonumber \\
& \geq\left(N^{\frac {l-(s-1)}{l+1}}e^{-t_{1}-t_{s+1}-\dotsb-t_{l}}\gamma h_{s-1}(|d_{0}|)^{-1}\right)^{1/s}\nonumber \\
& \geq N^{\frac{l-(s-1)}{l+1)s}}e^{-\frac{t_{1}+t_{s+1}+\dotsb +t_{l}}{s+1}-R^*_{s-1}(t)+O_{l}(\beta)},\nonumber 
\end{align}
{where we have used the fact that (\ref{eq:smalld0}) and (\ref{eq:123}) imply that
$${\gamma}h_{s-1}\left(|d_0|\right)^{-1}\geq \kappa h_{s-1}\left(|d_0|\right)^{-1}\geq e^{-sR_{s-1}(t)}\geq e^{-sR_{s-1}^*(t)+O_l(\beta)}.$$
}Then (\ref{eq:cases}) and (\ref{eq:lastmin}), or, {alternatively}, (\ref{eq:casesa10}) and (\ref{eq:lastmina10}), imply that
$$\lambda_{1}\lambda_{1}^{*}\geq e^{-2R^{*}(t)+O_{l}(\beta)},$$
proving $(ii)$. The last case that we have to consider is $d_{0}=0$. However, if $d_{0}=0$, for some $i\in\{1,\dotsc,l\}$ it must hold that
$${\|\bs{v}_{1}^{*}\|}\geq N^{-\frac1{l+1}}e^{t_{i}}\geq N^{-\frac1{l+1}}e^{R^{*}(t)};$$
hence, we arrive at $(iii)$.

To conclude the proof, we need to analyze cases $(i)$, $(ii)$, and $(iii)$. If $(i)$ occurs, we have two possibilities: either $\lambda_{1}\geq e^{-R^{*}(t)+O_{l}(\beta)}$ or $\lambda_{1}^{*}\geq e^{-R^{*}(t)+O_{l}(\beta)}$. If $\lambda_{1}\geq e^{-R^{*}(t)+O_{l}(\beta)}$, by (\ref{eq:counting}) and the trivial estimate $\lambda_{i}\geq \lambda_{1}^{i}$, we find
\begin{equation*}
N\ll_{l} N^{\frac l{l+1}}e^{(l+1)R^{*}(t)+O_{l}(\beta)},
\end{equation*}
whence
\begin{equation*}
N\ll_{l} e^{(l+1)^{2}R^{*}(t)+O_{l}(\beta)}.
\end{equation*}
Now, suppose that $\lambda_{1}^{*}\geq e^{-R^{*}(t)+O_{l}(\beta)}$. By Theorems \ref{thm:Min2} and \ref{thm:Ban}, and the fact that both the lattices $\Lambda_{\bs{t},N}$ and $\Lambda_{\bs{t},N}^{*}$ have determinant $1$, we deduce that
\begin{equation}
\label{eq:duality}
\lambda_{i}\asymp_{l}\delta_{1}\dotsm\delta_{i}\asymp_{l}\frac{1}{\delta_{i+1}\dotsm\delta_{l+1}}\asymp_{l}\delta_{1}^{*}\dotsm\delta_{l+1-i}^{*}\geq\left(\delta_{1}^{*}\right)^{l+1-i}=\left(\lambda_{1}^{*}\right)^{l+1-i},
\end{equation}
where $\delta_{i}$ and $\delta_{i}^{*}$ denote the $i$-th successive minimum of the lattices $\Lambda_{\bs{t},N}$ and $\Lambda_{\bs{t},N}^{*}$ respectively (see {Appendix} \ref{sec:AppA} for the notation).
Then, (\ref{eq:counting}) implies that
\begin{equation*}
N\ll_{l} \max_{i}\left\{N^{\frac i{l+1}}e^{2(l+1-i)R^{*}(t)+O_{l}(\beta)}\right\},
\end{equation*}
and thus we have
$$N\ll_{l} e^{2(l+1)R^{*}(t)+O_{l}(\beta)}.$$

Finally we need to consider cases $(ii)$ and $(iii)$. Case $(ii)$ is analogous to case $(i)$. If case $(iii)$ occurs, by (\ref{eq:duality}), we have
$$\lambda_{i}\gg_{l} \left\|\bs{v}_{1}^{*}\right\|^{l+1-i}\geq N^{-\frac{l+1-i}{l+1}}e^{(l+1-i)R^{*}(t)}.$$ 
Hence, (\ref{eq:counting}) implies that
$$N\ll_{l}1+\sum_{i=1}^{l}\frac{c_{2}N^{\frac i{l+1}}e^{(l+1-i)R^{*}(t)}}{\lambda_{i}}\ll_{l} 1+c_{2}N,$$
whence $N\ll_{l} 1$. The proof is therefore concluded.
\end{proof}

\subsection{Simultaneous Case}
We move on to discussing average counting in the simultaneous case. It will be useful to distinguish two cases: $t_{1}\geq 2lR(t)$ and $t_{1}<2lR(t)$, as the proof will differ substantially. We will refer to the case $t_{1}<2lR(t)$ as the simultaneous "degenerate" case. We remark once again, that, throughout the following subsection, we assume that the hypothesis of Corollary \ref{cor:RvsRstar} and of Lemma \ref{lem:deriv} holds, i.e., that $|\log\kappa|\leq e^{\beta}$.

\begin{lem}
\label{lem:alblockssim}
Let $\bs{t}$ be fixed with $T=t+t_{1}$ and $t_{1}\geq 2lR(t)$, and consider an interval $J\subset I_{0}$ of length $c_{2}e^{-T+(l+1)R(t)}$ with $c_{2}\geq e^{-(l+1)R(0)}$. Assume that there exists a dangerous interval $D_{\bs{t}}(S,b_{0},\bs{b})$ such that $J\cap D_{\bs{t}}(S,b_{0},\bs{b})$ contains at least one point that was not removed in the dual removing procedure, i.e., not lying in any dangerous interval of the form $D_{\bs{t}'}^{*}(S,b_{0}',\bs{b}')$ for any $\bs{t}'$, nor in the intervals $D_{\bs{t}'}^{*}(S',b_{0}',\bs{b}')$ or $D_{\bs{t}'}(S',b_{0}',\bs{b}')$ for any $S'$ with $\#S'< l$ and any multi-time $\bs{t}'$. Let
\begin{equation}
\mc{P}_{\bs{t}}(J,m):=\left\{(b_{0},\bs{b})\in\mb{Z}^{l+1}:D_{\bs{t}}(S,b_{0},\bs{b})\cap\bigcup_{i=0}^{m}M_{i}\neq\varnothing\right\},\nonumber
\end{equation}
where $M_{0}=J$, $|M_{i}|=|J|$ for all $i$, and the intervals $M_{i}$ and $M_{i+1}$ only share upper and lower endpoints respectively. Let also
$$m_{\bs{t}}(J):=\max\{m:\#\mc{P}_{\bs{t}}(J,i)\geq i\mbox{ for all }i\leq m\}$$
and
$$B_{\bs{t}}(J):=\bigcup_{i=0}^{m_{\bs{t}}(J)}M_{i}.$$
Then, if the constant $c_{2}$ is small enough in terms of $l$, we have that
$$m_{\bs{t}}(J)\leq e^{(l+1)^{2}R(t)+O_{l}(\beta)}.$$
\end{lem}

\begin{proof}
First, we observe that $N:=m_{\bs{t}}(J)$ is well defined, since $$\sup_{m}\#\mc{P}_{\bs{t}}(J,m)\leq |I_{0}|e^{2t-2R(t)}$$ (this follows from the fact that once $b_{0}$ is fixed, also $b_{i}$ for $i\geq 2$ are fixed and $b_{1}$ can assume at most $|I_{0}|e^{t-R(t)}$ values). Moreover, by the definition of $m_{\bs{t}}(J)$, we have that 
$$\#\mc{P}_{\bs{t}}(J,N)=N.$$
Let $y\in J$ be a point lying in some dangerous interval that was not removed in the dual removing procedure. Then for any $(b_{0},\bs{b})\in\mc{P}_{\bs{t}}(J,N)$ and any $x\in D_{\bs{t}}(S,b_{0},\bs{b})$ we have
\begin{multline}
\label{eq:oneside*1}
|b_{1}+b_{0}f_{1}(y)|\leq |b_{1}+b_{0}f_{1}(x)|+|b_{0}f_{1}'(x)||x-y| \\
\leq e^{-t_{1}-R(t)}+e^{t-R(t)}Nc_{2}e^{-(t+t_{1})+(l+1)R(t)}\leq 2Nc_{2}e^{-t_{1}+lR(t)},
\end{multline}
where we used the fact that $c_{2}\geq e^{-(l+1)R(t)}$. 

Let $D_{N}':=\textup{diag}\left(N^{\frac1{l+1}},N^{-\frac l{l+1}},N^{\frac1{l+1}},N^{\frac1{l+1}}\right)$. To any integer vector $(b_{0},\bs{b})\in\mc{P}_{\bs{t}}(J,N)$, we uniquely associate the vector 
\begin{equation}
\label{eq:v1}
\left(N^{\frac1{l+1}}e^{-t}b_{0},N^{-\frac l{l+1}}e^{t_{1}}(b_{1}+b_{0}f_{1}(y)),\dotsc,N^{\frac1{l+1}}e^{t_{l}}(b_{l}+b_{0}f_{l}(y))\right).
\end{equation}
lying in the lattice
$$\Lambda_{\bs{t},N}'(y):=D_{N}'a (-t,-\bs{t})\Lambda (y)^{\scriptscriptstyle{T}},$$
where $\Lambda(y)$ was introduced in the proof of Lemma \ref{lem:alblocksdual}. We also define
\begin{equation*}
\mc{B}(N)':=\left[e^{t-R(t)},e^{t-R(t)}\right]
\times\left[-2c_{2}Ne^{-t_{1}+lR(t)},2c_{2}Ne^{-t_{1}+lR(t)}\right]
\times\prod_{i=2}^{l}\left[e^{-t_{i}-R(t)},e^{-t_{i}-R(t)}\right].\nonumber
\end{equation*}
Then, by (\ref{eq:oneside*1}) and the definition of dangerous set, we have
\begin{equation}
N\leq \#\left(\Lambda_{\bs{t},N}'(y)\cap D_{N}'a (-t,-\bs{t})\mc{B}'(N)\right).\nonumber
\end{equation}
Since $\textup{Vol}\left(\mc{B}(N)'\right)\leq 4c_{2}N$, by the same argument as used in Lemma \ref{lem:alblocksdual} (provided the constant $c_{2}$ is sufficiently small in terms of $l$), it is enough to assume that the vectors in $\mc{P}_{\bs{t}}(J,N)$ lie on a hyperplane. Then we deduce from Corollary \ref{cor:counting} that
\begin{equation}
\label{eq:countingsim}
N\ll_{l}1+\frac{c_{2}N^{\frac1{l+1}}e^{lR(t)}}{\lambda_{1}}+\dotsb+\frac{c_{2}N^{\frac l{l+1}}e^{R(t)}}{\lambda_{l}},
\end{equation}
where $\lambda_{i}$ denotes the first minimum of the lattice $\bigwedge^{i}{\Lambda'_{\bs{t},N}(y)}$ for $i=1,\dotsc,l$, and we used the fact that $c_{2}e^{lR(t)}\geq e^{-R(t)}$. This time, we will estimate $N$ by analyzing only the first minimum $\lambda_{1}^{*}$ of the lattice
$$\left(\Lambda_{\bs{t},N}'\right)^{*}:={\left(\left(\Lambda_{\bs{t},N}'\right)^{-1}\right)^{\scriptscriptstyle{T}}}.$$
We aim to show that one of the following two cases holds
\begin{itemize}
\item[$i)$]$\lambda_{1}^{*}\geq e^{-lR(t)+O_{l}(\beta)}$;\vspace{2mm}
\item[$ii)$]$\lambda_{1}^{*}\geq N^{-\frac1{l+1}}e^{R(t)}$.
\end{itemize}

Before proceeding to the proof, let us clarify why the minimum $\lambda_{1}$ cannot be dealt with directly. Consider a vector $\bs{v}_{1}$ as in (\ref{eq:v1}) such that $|\bs{v}_{1}|=\lambda_{1}$, and assume that $b_{0}\neq 0$. If we were to proceed as in Lemma \ref{lem:alblocksdual}, we would apply (\ref{eq:AM-GM}) to the entries of the vector $\bs{v}_{1}$ and estimate the product $b_{0}|b_{1}+b_{0}f_{1}(y)|\dotsm|b_{l}+b_{0}f_{l}(y)|$ from below through Lemma \ref{lem:siminductivehyp}. However, while one can easily assume that $|b_0|<e^{t-R(t)}+O_l(\beta)$, it is not clear how to reduce to the case $|b_{1}+b_{0}f_{1}(y)|<e^{-t_{1}-R(t)+O_{l}(\beta)}$. Hence, it could happen that, e.g., $y$ lies in a simultaneous dangerous interval with $t'<t$ but $t'+t_{1}'\geq T$.
On the other hand, our inductive hypothesis only guarantees that we have removed all dangerous intervals $D_{\bs{t}'}(S,b_{0},\bs{b})$ for $t'+t_{1}'<T$ and thus, we cannot recover information about the product $b_{0}|b_{1}+b_{0}f_{1}(y)|\dotsm|b_{l}+b_{0}f_{l}(y)|$ from the simultaneous inductive hypothesis. 

We therefore start directly by analyzing the minimum $\lambda_{1}^{*}$. We choose a vector $\bs{v}_{1}^{*}$ whose length is equal to $\lambda_{1}^{*}$. It will have the form
\begin{equation*}
\bs{v}_{1}^{*}:=\left(N^{-\frac1{l+1}}e^{t}\tilde{f}_{(d_{0},\bs{d})}(y),N^{\frac l{l+1}}e^{-t_{1}}d_{1},N^{-\frac1{l+1}}e^{-t_{2}}d_{2},\dotsc,N^{-\frac1{l+1}}e^{-t_{l}}d_{l}\right)
\end{equation*}
for some $(d_{0},\bs{d})\in\mb{Z}^{l+1}$, where
$$\tilde{f}_{(d_{0},\bs{d})}(x):=d_{0}-\sum_{i=1}^{l}d_{i}f_{i}(y).$$
If $\bs{d}=\bs{0}$, we have $\lambda_{1}^{*}=N^{-\frac1{l+1}}e^{t}$ and $(ii)$ occurs (recall that one can always assume that $t-nR(t)\geq 0$). If $\bs{d}\neq \bs{0}$ we have two possible cases: either there exists $i\in\{1,\dotsc,l\}$ for which $|d_{i}|\geq e^{t_{i}+R(t)}$ and hence $(ii)$, or for all $i=1,\dotsc,d$ we have that
\begin{equation}
\label{eq:smalldi}
|d_{i}|< e^{t_{i}+R(t)}.
\end{equation}
In the latter case, we must make a further distinction. If $d_{1}\neq 0$, we may assume without loss of generality that $d_{1},\dotsc,d_{s}\neq 0$ for some $s\leq l$, and that the remaining components of the vector $\bs{d}$ are null. Then we choose $t'$ such that
$$t'-sR_s^*(t')=\sum_{i=1}^s t_i+sR_l(t)+2\beta s^2$$
and we set $t_i':=t_i+R_l^*(t)+2\beta s+R_{s}(t')$ for $i=1,\dotsc,s$. With this choice, we have that $|d_i|<e^{t_{i}'-R_{s}^*(t')-2\beta s}$ for $i=1,\dotsc, s$.
Then, from (\ref{eq:AM-GM}), Corollary \ref{cor:RvsRstar}, and Lemma \ref{lem:dualinductivehyp}, applied to the set $S'=\{2,\dotsc,s\}$ and $T'>t'+t_1'$, we obtain
\begin{multline*}
\lambda_{1}^{*}={\|\bs{v}_{1}^{*}\|}\geq\left(N^{\frac{l-s}{l+1}}e^{t_{s+1}+\dotsb +t_{l}}|\tilde{f}_{(d_{0},\bs{d})}(y)||d_{1}|\dotsm|d_{s}|\right)^{\frac 1{s+1}} \\
 \gg_{l}e^{-\frac{(l-s)R^*(t)}{s+1}-R^*_{s}(t')}\geq e^{-lR^*(t+sR^*(t)+sR(t)+O_{l}(\beta))}\geq e^{-lR(t)+O_{l}(\beta)},
\end{multline*}
where we used $t'\geq t$ and Lemma \ref{lem:deriv}.

If $d_{1}=0$, we may assume that $d_{2},\dotsc,d_{s}\neq 0$ with $2\leq s\leq l$, and that the remaining components of the vector $\bs{d}$ are null. Then, by (\ref{eq:AM-GM}), we find that
\begin{align*}
& \lambda_{1}^{*}={\|\bs{v}_{1}^{*}\|}\geq\left(N^{-\frac s{l+1}}e^{t_{1}+t_{s+1}+\dotsb +t_{l}}|\tilde{f}_{(d_{0},\bs{d})}(y)||d_{2}|\dotsm|d_{s}|\right)^{1/s} \\
& \geq \left(N^{-\frac s{l+1}}e^{t_{1}+t_{s+1}+\dotsb +t_{l}}{\gamma}h_{s-1}(|d_{2}|\dotsm|d_{s}|)\right)^{1/s} \geq N^{-\frac1{l+1}}e^{t_{1}/s-(l-s)R(t)/s-R_{s}^*(t+sR(t))} \\
& \geq N^{-\frac1{l+1}}e^{t_{1}/s-lR^*(t+sR(t))/s}\geq N^{-\frac1{l+1}}e^{R(t)+O_{l}(\beta)},
\end{align*}
where we used $t_{1}\geq 2lR(t)$ and the fact that (\ref{eq:smalldi}) and (\ref{eq:123}) imply that
$${\gamma}h_{s-1}\left(|d_{2}|\dotsm|d_{s}|\right)^{-1}\geq \kappa h_{s-1}\left(|d_{2}|\dotsm|d_{s}|\right)^{-1}\geq e^{-sR_{s-1}^*(t+sR(t))}\geq e^{-sR(t)+O_{l}(\beta)}.$$
Note that we once again used Lemma \ref{lem:deriv} in the last inequality. This proves $(ii)$.

The conclusion in case $(ii)$ is analogous to that of Lemma \ref{lem:alblocksdual}. In case $(i)$, on the other hand, from (\ref{eq:countingsim}) we deduce that
$$N\ll_{l} \max_{i} N^{\frac i{l+1}}e^{(l+1)(l+1-i)R(t)+O_{l}(\beta)},$$
whence
$$N\ll_{l} e^{(l+1)^{2}R(t)+O_{l}(\beta)}.$$
\end{proof}

\begin{lem}[{The} "degenerate" case]
\label{lem:alblockssim2}
Let $\bs{t}$ be fixed with $T=t+t_{1}$ and $t_{1}< 2lR(t)$. Let $J\subset I_{0}$ be an intyerval of length $c_{2}e^{-T+(l+1)R(t)}$ with $$c_{2}\geq e^{-(l+1)R(0)}\quad\text{and}\quad c_{2}/4>e^{-\beta}.$$ Assume that there exists a dangerous interval $D_{\bs{t}}(S,b_{0},\bs{b})$ such that $J\cap D_{\bs{t}}(S,b_{0},\bs{b})$ contains at least one point that was not removed in the removing procedure for sets smaller than $S$, i.e., not lying in any dangerous interval of the form $D_{\bs{t}'}^{*}(S',b_{0}',\bs{b}')$ or $D_{\bs{t}'}(S',b_{0}',\bs{b}')$ for any $S'$ with $\#S'< l$ and any multi-time $\bs{t}'$. Let
\begin{equation}
\mc{P}_{\bs{t}}\left(J,m\right):=\left\{(b_{0},\bs{b}): D_{\bs{t}}(S,b_{0},\bs{b})\cap\bigcup_{i=0}^{m}M_{i}\neq\varnothing\right\},\nonumber
\end{equation}
where $M_{0}=J$, $|M_{i}|=\left|J\right|$ for all $i$, and $M_{i}$ and $M_{i+1}$ only share upper and lower endpoints respectively. Let also
$$m_{\bs{t}}\left(J\right):=\max\left\{m:\#\mc{P}_{\bs{t}}\left(J,i\right)\geq i\mbox{ for all }i\leq m\right\}$$
and
$$B_{\bs{t}}\left(J\right):=\bigcup_{i=0}^{m_{\bs{t}}\left(J\right)}M_{i}.$$
Then, if the constant $c_{2}$ is small enough in terms of $l$, we have
$$m_{\bs{t}}\left(J\right)\leq e^{(l+1)(l+2)(4l+3)R(t)+O_{l}(\beta)}.$$
\end{lem}

\begin{proof}
Once again, the quantity $N:=m_{\bs{t}}\left(J\right)$ is well-defined, since $$\sup_{m}\#\mc{P}_{\bs{t}}\left(J,m\right)\leq |I_{0}|e^{2t-2R(t)}.$$ Assume by contradiction that
$N\geq Ae^{(l+1)(l+2)(4l+3)R(t)}$ for some constant $A\geq 1$ (only depending on $l$) yet to be determined. Since $N>2l$, the volume argument used in Lemma \ref{lem:alblockssim} shows that the points in the set $\mc{P}_{\bs{t}}\left(J,N\right)$ must lie on a hyperplane. Denote this hyperplane by $\pi$.
To prove the claim, we introduce a second interval $\tilde{J}$ of length $c_{2}e^{-T+(2l+3)R(t)}$ whose lower endpoint coincides with the lower endpoint of $J$, and we repeat the construction above starting from $\tilde{J}$ and, this time, remembering the hyperplane $\pi$. More precisely, we set
\begin{equation}
\mc{P}_{\bs{t}}\left(\tilde{J},\pi,m\right):=\left\{(b_{0},\bs{b})\in\pi: D_{\bs{t}}(S,b_{0},\bs{b})\cap\bigcup_{i=0}^{m}\tilde{M}_{i}\neq\varnothing\right\},\nonumber
\end{equation}
where $\tilde{M}_{0}=\tilde{J}$, $\left|\tilde{M}_{i}\right|=\left|\tilde{J}\right|$ for all $i$, and the intervals $\tilde{M}_{i}$ and $\tilde{M}_{i+1}$ only share upper and lower endpoints respectively. We also let
$$\tilde{N}=m_{\bs{t}}\left(\pi,\tilde{J}\right):=\max\left\{m:\#\mc{P}_{\bs{t}}\left(\tilde{J},\pi,i\right)\geq i\mbox{ for all }i\leq m\right\}$$
and
$$B_{\bs{t}}\left(\tilde{J},\pi\right):=\bigcup_{i=0}^{\tilde{N}}\tilde{M}_{i}.$$
Then, by construction, it must be $\tilde{N}=\#\mc{P}_{\bs{t}}\left(\tilde{J},\pi,\tilde{N}\right)$.

We will need the following auxiliary result, the proof of which we postpone to the next subsection.

\begin{lem}
\label{lem:directionofpi}
If $\tilde{N}\geq 2$, we have that $\bs{e}_{1}=(0,1,0\dotsc,0)\notin\pi$.
\end{lem}

Further, we fix a point $y\in J$ in some dangerous interval that was not removed in previous steps. Then for any $(b_{0},\bs{b})\in\mc{P}_{\bs{t}}\left(\tilde{J},\pi,\tilde{N}\right)$ and any $x\in D_{\bs{t}}(S,b_{0},\bs{b})$ we have
\begin{equation}
\label{eq:oneside*2}
\begin{aligned}
|b_{1}+b_{0}f_{1}(y)|&\leq |b_{1}+b_{0}f_{1}(x)|+|b_{0}f_{1}'(x)|\cdot|x-y| \\
&\leq e^{-t_{1}-R(t)}+ e^{t-R(t)}\tilde{N}c_{2}e^{-(t+t_{1})+(2l+3)R(t)}\leq 2\tilde{N}c_{2}e^{-t_{1}+(2l+2)R(t)}.
\end{aligned}\end{equation}
As in the proof of Lemma \ref{lem:alblockssim}, to any integer vector $(b_{0},\bs{b})\in\mc{P}_{\bs{t}}\left(\tilde{J},\pi,\tilde{N}\right)$, we uniquely associate the vector 
\begin{equation}
\label{eq:v1tilde}
\left(\tilde{N}^{\frac1{l+1}}e^{-t}b_{0},\tilde{N}^{-\frac l{l+1}}e^{t_{1}}(b_{1}+b_{0}f_{1}(y)),\dotsc,\tilde{N}^{\frac1{l+1}}e^{t_{l}}(b_{l}+b_{0}f_{l}(y))\right),
\end{equation}
lying in a codimension $1$ sub-lattice of the lattice
$$\Lambda_{\bs{t},\pi,\tilde{N}}'(y):=D_{\tilde{N}}'a (-t,-\bs{t})\Lambda (y)^{\scriptscriptstyle{T}},$$
and  define
\begin{multline*}
\mc{B}'(\tilde{N}):=\left[e^{t-R(t)},e^{t-R(t)}\right]
\times\left[-2c_{2}\tilde{N}e^{-t_{1}+(2l+2)R(t)},2c_{2}\tilde{N}e^{-t_{1}+(2l+2)R(t)}\right]
 \\
 \times\prod_{i=2}^{l}\left[e^{-t_{i}-R(t)},e^{-t_{i}-R(t)}\right].\nonumber
\end{multline*}
Then, by (\ref{eq:oneside*2}) and the definition of dangerous set, we have that
\begin{equation}
\tilde{N}\leq \#\left(\Lambda_{\bs{t},\pi,\tilde{N}}'(y)\cap D_{\tilde{N}}'a (-t,-\bs{t})\mc{B}'(\tilde{N})\right).\nonumber
\end{equation}
Since the vectors $(b_{0},\bs{b})$ lie on the hyperplane $\pi$, we deduce from Corollary \ref{cor:counting} that
\begin{equation}
\label{eq:countingsimdeg}
\tilde{N}\ll_{l}1+\frac{c_{2}\tilde{N}^{\frac1{l+1}}e^{(2l+2)R(t)}}{\lambda_{1}}+\dotsb+\frac{c_{2}\tilde{N}^{\frac l{l+1}} e^{(l+3)R(t)}}{{\lambda}_{l}},
\end{equation}
where $\lambda_{i}$ denotes the first minimum of the lattice $\bigwedge^{i}\Lambda_{\bs{t},\pi,\tilde{N}}'(y)$ for $i=1,\dotsc,l$ (note that we used the fact that $c_{2}e^{(2l+2)R(t)}\geq e^{-R(t)}$). We will show that, this time (in contrast with Lemma \ref{lem:alblockssim}), the only case to occur is $\lambda_{1}\geq e^{-3R(t)}$.

Let $\bs{v}_{1}$ be as in (\ref{eq:v1tilde}) and assume that $|\bs{v}_{1}|=\lambda_{1}$. If $b_{0}\neq 0$ and \begin{equation}
\label{eq:smallaideg}
|b_{0}|\leq e^{t-R(t)}
\end{equation}
we apply (\ref{eq:AM-GM}) excluding the component $1$ in $\bs{v}_{1}$. Assuming $y$ irrational, we conclude that
\begin{multline}
\lambda_{1}={\|\bs{v}_{1}\|}\geq \left(\tilde{N}^{\frac l{l+1}}e^{-t_{1}}|b_{0}||b_{2}+b_{0}f_{2}(y)|\dotsm|b_{l}+b_{0}f_{l}(y)|\right)^{1/l} \\
\geq\left(\tilde{N}^{\frac l{l+1}}e^{-t_{1}}{\gamma}h_{l-1}(|b_{0}|)^{-1}\right)^{1/l} \geq e^{-3R(t)},\nonumber
\end{multline}
where we used the fact that (\ref{eq:smallaideg}) and (\ref{eq:123}) imply that
$${\gamma}h_{l-1}\left(|b_{0}|\right)^{-1}\geq \kappa h_{l-1}\left(|b_{0}|\right)^{-1}\geq e^{-lR_{l-1}(t)},$$
and the fact that $t_{1}\leq 2l$.


If $|b_{0}|\geq e^{t-R(t)}$, then
$$\lambda_{1}={\|\bs{v}_{1}\|}\geq \tilde{N}^{\frac1{l+1}}e^{-t}|b_{0}|\geq e^{-R(t)},$$
and once again the claim is proved.
Finally, if $b_{0}=0$, then either $b_{i}\neq 0$ for some $i\geq 2$, and hence $\lambda_{1}\geq e^{-R(t)}$, or $b_{0}=b_{2}=\dotsb b_{l}=0$. This case, however, is excluded by Lemma \ref{lem:directionofpi}.

From (\ref{eq:countingsimdeg}) and the fact that $\lambda_{1}\geq e^{-3R(t)}$, we obtain that
$$\tilde{N}\ll_{l} \max_{i}\tilde{N}^{\frac i{l+1}}e^{(2l+3-i)R(t)+3iR(t)},$$
whence
$$\tilde{N}\ll_{l} e^{(l+1)(4l+3)R(t)}.$$
Note that if the hypothesis of Lemma \ref{lem:directionofpi} is not satisfied, i.e., if $\tilde{N}<2$, the previous inequality still holds. Now, since we assumed that $N\geq Ae^{(l+1)(l+2)(4l+3)R(t)}$ and that $\left|\tilde{J}\right|=e^{(l+2)R(t)}|J|$, if the constant $A$ is large enough in terms of $l$, we find that
\begin{equation}
\label{eq:containment2}
B_{\bs{t}}(\tilde{J},\pi)\subset B_{\bs{t}}\left(J\right).
\end{equation}
Hence, all integer vectors $(b_{0},\bs{b})$ such that $D_{\bs{t}}(S,b_{0},\bs{b})\cap B_{\bs{t}}(\tilde{J},\pi)\neq \varnothing$ are also contained in $\mc{P}_{\bs{t}}\left(J,N\right)$ and thus lie on $\pi$. This shows that $\tilde{N}$ bounds not just the number of integer vectors on $\pi$ whose dangerous interval intersects $B_{\bs{t}}(\tilde{J},\pi)$, but the number of \emph{all} integer vectors whose dangerous interval intersects the block $B_{\bs{t}}(\tilde{J},\pi)$. However, $B_{\bs{t}}(\tilde{J},\pi)$ contains at least $\left\lfloor\tilde{N}e^{(l+2)R(t)}\right\rfloor>\tilde{N}$ intervals of length $|J|$, while being intersected by only $\tilde{N}$ dangerous intervals $D_{\bs{t}}(S,b_{0},\bs{b})$. This contradicts (\ref{eq:containment2}) and the definition of $N=m_{\bs{t}}\left(J\right)$, completing the proof.
\end{proof}

\subsection{Proof of Lemma \ref{lem:directionofpi}}
Let us pick $(b_{0},\bs{b})\in\pi$ such that $y\in D_{\bs{t}}(S,b_{0},\bs{b})\cap J$. By Lemma \ref{lem:danlength}, we may assume that $|b_{0}|\geq e^{t-R(t)-\beta}$. Suppose by contradiction that $\bs{e}_{1}\in\pi$. Then the vectors $P_{k}:=(b_{0},b_{1}+k,b_{2},\dotsc,b_{l})$ lie in $\pi$ for arbitrary values of $k\in\mb{Z}$. For each $k\in\mb{Z}$ let us choose $x_{k}\in\mb{R}$ such that
$$b_{0}(x_{k}-y)=-k.$$
Then we have 
$$|b_{1}+k+b_{0}f_{1}(x_{k})|=|b_{1}+b_{0}y+k+b_{0}(x_{k}-y)|=|b_{1}+b_{0}f_{1}(y)|<e^{-t_{1}-R(t)}.$$
Hence, if $x_{k}\in I_{0}$, we have that $x_{k}\in D_{\bs{t}}(S,P_{k})$. This, along with $P_{k}\in\pi$, implies that
\begin{equation}
\label{eq:containment}
\left\{P_{k}:x_{k}\in B_{\bs{t}}(\tilde{J},\pi)\right\}\subset \mc{P}_{\bs{t}}\left(\tilde{J},\pi,\tilde{N}\right).
\end{equation}
Further, we observe that $x_{k}\in B_{\bs{t}}(\tilde{J},\pi)$ whenever
$$-\frac{k}{b_{0}}=x_{k}-y\in \left[0,(\tilde{N}-1)c_{2}e^{-T+(2l+3)R(t)}\right).$$
This happens for
$$\textup{sgn}(b_{0})k\in \left[-(\tilde{N}-1)c_{2}e^{-t_{1}+(2l+2)R(t)-\beta},0\right),$$
i.e., for at least $\tilde{N}c_{2}e^{2R(t)-\beta}/4\geq\tilde{N}c_{2}e^{\beta}/4>\tilde{N}$ values of $k$. Here, we used the fact that $R(0)\geq \beta$, as a consequence of $\kappa\leq e^{-(l+1)\beta}$ (assumed in Lemma \ref{lem:danlength}) and the fact that $t_1\leq 2lR(t)$. This contradicts (\ref{eq:containment}) and the assumption that $\#\mc{P}_{\bs{t}}\left(\tilde{J},\pi,\tilde{N}\right)=\tilde{N}$.

\section{Proof of Proposition \ref{prop:complementCR}}
\label{sec:conclusion}

{Throughout this section, we will denote dangerous intervals by $D_{\bs{t}}^{(*)}(b_{0},\bs{b})$, as there is no distinction between the dual and {the} simultaneous case. The assumptions of Section \ref{sec:setup} will be in place. We will also be using the symbol $R^{(*)}$ to denote either one of the function $R$ or $R^{*}$. There is no significant difference in the remaining part of the proof. For example, in the conclusion of Lemmas \ref{lem:alblocksdual}, \ref{lem:alblockssim}, and \ref{lem:alblockssim2}, the functions $R$ and $R^*$ are interchangable, due to Corollary \ref{cor:RvsRstar}}. 

In analogy with \cite{Bad}, we choose $r_{k}:=e^{\beta} k\log k$ for $k\geq 2$ and $r_{k}=1$ otherwise. Then we have
$$F(k)=e^{k\beta}k!\prod_{i=1}^{k}\log^{+}i$$
and
\begin{equation}
\label{eq:logFn}
k\beta\leq \log F(k)\leq k\beta + 2\log(k!)\leq k(2\log k+\beta).
\end{equation}
As discussed in Section \ref{sec:setup}, for fixed $k$ we remove all intervals $I_{k}\in\mc{I}_{k-1}/r_{k-1}$ that intersect dangerous intervals $D_{\bs{t}}^{1(*)}(b_{0},\bs{b})$ for which the parameter $T=t+t_{1}$ satisfies 
\begin{equation}
\label{eq:Trange1}
L^{-1}F(k-2)\leq e^{T}<L^{-1}F(k-1).
\end{equation}

\begin{lem}
\label{lem:computations}
Assume that $4\max\big\{|\log\kappa|,l\log\beta\big\}\leq e^{\beta}$. Then for all $k\geq 2$ and $T>0$ satisfying \eqref{eq:Trange1} and all values of $\bs t\in C_R^{(*)}\cap \beta\mb Z^{l}$ such that $T=t+t_1$ and $t\geq 0$, we have
\begin{equation*}
 \frac{\kappa^{-1}T^{l-1}\log k}{e^{l\beta}\log\log L}\ll_{l} e^{(l+1)R^{(*)}(t)}\ll_{l} e^{l\beta}\kappa^{-1}(k\log k)^{l-1}\log k.   
\end{equation*}
\end{lem}
\noindent We will prove Lemma \ref{lem:computations} at the end of this section.

Our first goal will be to choose the parameter $p=p(k)$ as outlined in Section \ref{sec:setup}. From Lemmas \ref{lem:alblocksdual}, \ref{lem:alblockssim}, and \ref{lem:alblockssim2}, it follows that any block $B_{\bs{t}}^{(*)}(J)$ constructed over an interval $J\subset I_{0}$ of length $c_{2}e^{-T+lR^{(*)}(t)}$ has length at most
\begin{equation}
\left|B_{\bs{t}}^{(*)}(J)\right|\leq m_{\bs{t}}^{(*)}(J)|J|\leq e^{B_{l}\beta} e^{C_{l}(l+1)R^{(*)}(t)}c_{2}e^{-T+(l+1)R^{(*)}(t)}\nonumber
\end{equation}
where $C_{l}:=(l+2)(4l+3)$ and $B_{l}$ is some fixed constant depending on $l$ deriving from the error term $O_{l}(\beta)$ in Lemmas \ref{lem:alblocksdual}, \ref{lem:alblockssim}, and \ref{lem:alblockssim2}, and from Corollary \ref{cor:RvsRstar}. By doubling the constant $B_{l}$ and assuming that it is large enough in terms of $l$, we may absorb the term $e^{l\beta}$ and the constants depending on $l$ in Lemma \ref{lem:computations}, thus obtaining
\begin{multline}
\label{eq:length}
\left|B_{\bs{t}}^{(*)}(J)\right|\leq e^{B_{l}\beta} e^{C_{l}(l+1)R^{(*)}(t)}c_{2}e^{-T+(l+1)R^{(*)}(t)} \\
\leq e^{2B_l\beta}\left(\kappa^{-1}k^{l-1}(\log k)^{l}\right)^{C_l+1}\min\left\{1,\frac{L}{F(k-2)}\right\},
\end{multline}
where the minimum stems from the fact that $T$ can always be assumed positive.

Now, we aim to choose $p=p(k)$ so that any interval $I_{p}\in\mc{I}_{p}$ fits at least one block $B_{\bs{t}}^{(*)}(J)$. To this end, we set $A_{l}:=100l(C_{l}+1)+2B_{l}+2$ and we observe that for $k\geq 2A_{l}$ we have
$$\frac{(k-A_{l})\log(k-A_{l})}{k\log k}\geq \left(1-\frac{A_{l}}{k}\right)\left(1-\frac{\log 2}{\log k}\right)\geq \frac{1}{4}.$$
Therefore, assuming that
\begin{equation}
\label{eq:cond2}
\left(\frac{e^\beta}{4}\right)^{(A_l-2)}\geq e^{2B_l\beta}\kappa^{-(C_l+1)},
\end{equation}
we can conclude that
\begin{multline}
\label{eq:fitting1}
\frac{L}{F(k-A_{l})}\geq e^{(A_l-2)\beta}\frac{\left(k\log k\right)^{A_l-2}}{4^{A_l-2}}\frac{L}{F(k-2)} \\
\geq e^{2B_{l}\beta}\left(\kappa^{-1}k^{l-1}(\log k)^{l}\right)^{C_{l}+1}\frac{L}{F(k-2)}.
\end{multline}

In view of (\ref{eq:length}) and (\ref{eq:fitting1}), for $k\geq 2A_{l}$ we define $p(k):=k-A_{l}+1$, so that each interval $I_{p}\in\mc{I}_{p}$ contains at least one block $B_{\bs{t}}^{(*)}(J)$. Further, we impose
\begin{equation}
\label{eq:longI0}
L\geq e^{2B_{l}\beta}\left(\kappa^{-1}(2A_{l})^{l-1}\log (2A_{l})^{l}\right)^{C_{l}+1}. \end{equation}
By (\ref{eq:length}), this condition ensures that for $k<2A_{l}$, the interval $I_{0}$ fits at least one block $B_{\bs{t}}^{(*)}(J)$ for any time $\bs{t}$, with $T=t+t_1$ satisfying (\ref{eq:Trange1}). Then for $k<2A_{l}$ we set $p(k)=0$.
With these definitions, we have
\begin{equation}
\label{eq:pn}
|I_{p}|=\begin{cases}
\dfrac{L}{F(k-A_{l})} & \mbox{if }k\geq 2A_{l} \\
L & \mbox{if }k<2A_{l}
\end{cases}.
\end{equation}

Now, we proceed to analyze the term $\#\left\{(b_{0},\bs{b})\in\mb{Z}^{l+1}:D_{\bs{t}}^{(*)}(b_{0},\bs{b})\cap I_{p}\neq\varnothing\right\}$ in (\ref{eq:finalform}) for a fixed $\bs{t}$. We subdivide any interval $I_{p}\in\mc{I}_{p}$ as in (\ref{eq:pn}) into sub-intervals of length $c_{2}e^{-T+(l+1)R^{(*)}(t)}$ and we assemble them in disjoint blocks as in Lemmas \ref{lem:alblocksdual}, \ref{lem:alblockssim}, and \ref{lem:alblockssim2}, considering as a block (formed by one single interval) also those intervals that are not part of any previous block and that are not intersected by any set $D_{\bs{t}}^{(*)}(S,b_{0},\bs{b})$, or that have been entirely removed in previous steps. By Lemmas \ref{lem:alblocksdual}, \ref{lem:alblockssim}, and \ref{lem:alblockssim2}, we have that the union of all those blocks that are properly contained in a fixed interval $I_{p}$ is intersected by as many dangerous intervals as the number of intervals $J$ contained in it. This number is at most
\begin{equation*}
|I_{p}|c_{2}^{-1}e^{T-(l+1)R^{(*)}(t)}\geq 1.
\end{equation*}

We must now take into account the last block, which may not be entirely contained in the interval $I_{p}$. By Lemmas \ref{lem:alblocksdual}, \ref{lem:alblockssim}, and \ref{lem:alblockssim2}, this block is formed by at most $e^{B_{l}}e^{C_{l}(l+1)R^{(*)}(t)}$ 
intervals $J$. However, by (\ref{eq:length}), (\ref{eq:fitting1}), and (\ref{eq:longI0}) we have that
$$e^{B_{l}\beta}e^{C_{l}(l+1)R^{(*)}(t)}|J|\leq |I_{p}|$$
for $p=p(k)$, whence
$$\frac{|I_{p}|}{|J|}\geq e^{B_{l}\beta}e^{C_{l}(l+1)R^{(*)}(t)}.$$
Thus, the number of dangerous intervals intersecting the last block constructed over $I_{p}$ is comparable to the number of dangerous intervals intersecting the blocks well inside $I_{p}$. In view of this, we deduce that for fixed $\bs{t}$ it holds
\begin{equation}
\label{eq:number1}
\#\left\{(b_{0},\bs{b})\in\mb{Z}^{l+1}:D_{\bs{t}}^{(*)}(b_{0},\bs{b})\cap I_{p}\neq\varnothing\right\}\leq 2|I_{p}|c_{2}^{-1}e^{T-(l+1)R^{(*)}(t)}.
\end{equation}

In what follows, we continue to work under the assumption that $\max\{|\log\kappa|,l\log\beta\}\leq e^{\beta}$, as in Lemma \ref{lem:computations}. We are now left to sum over the sets $\mc{D}(k)$ and $\mc{S}(k)$ in (\ref{eq:finalform}). We do this in two steps. First, we consider all possible vectors $\bs{t}$ in $\mc{D}(k)$ and $\mc{S}(k)$ that give raise to the same $T=2t_{1}+\sum_{i=2}^{l}t_{i}$. From (\ref{eq:tvsRt}) we have that, if $t\geq 4\max\{|\log\kappa|,l\log\beta\}$, then $T=t+t_{1}\geq t-R^{(*)}(t)\geq t/2$. Hence, given that
$$|t_{i}|\leq t+(l-1)R^{(*)}(t)\leq lt\leq \max\{2lT,4l\max\{|\log\kappa|,l\log\beta\}\}$$ for all $i$, we have 
$$\#\big\{\bs{t}\in\mc{D}(k)\cup\mc{S}(k):t+t_{1}=T\big\}\leq D_{l}e^{(l-1)\beta}T^{l-1},$$
with $D_{l}$ a constant only depending on $l$.
Then, by (\ref{eq:number1}) and Lemma \ref{lem:computations}, the number of different intervals $D_{\bs{t}}^{(*)}(b_{0},\bs{b})$ that are removed from any $I_{p}\in\mc{I}_{p}$ for a fixed $T$ is bounded above by
\begin{equation}
\label{eq:numbD}
E_{l}\cdot e^{(l-1)\beta}T^{l-1}\cdot |I_{p}|c_{2}^{-1}e^{T}\cdot \frac{e^{l\beta}\log\log L}{\kappa^{-1}T^{l-1}\log k},
\end{equation}
with $E_{l}$ a new constant only depending on $l$. Further, by Lemmas \ref{lem:danlengthdual} and \ref{lem:danlength}, and by (\ref{eq:Trange1}), each set $D_{\bs{t}}^{(*)}(b_{0},\bs{b})$ removes as many as
\begin{equation}
\label{eq:numbIq}
2\frac{e^{-T+\beta}}{|I_{k}|}
\end{equation} 
intervals from $\mc{I}_{k-1}/r_{k-1}$.
Combining (\ref{eq:numbD}) and (\ref{eq:numbIq}), we find that the total number of intervals $I_{k}$ removed from any interval $I_{p}\in\mc{I}_{p}$ for a fixed time $T$ is bounded above by
\begin{equation}
2E_le^{2l\beta}\kappa c_2^{-1}\frac{\log\log L}{\log k}\frac{|I_p|}{|I_k|}.\nonumber
\end{equation}
Since we assumed that $L^{-1}F(k-2)\leq e^{T}<L^{-1}F(k-1)$, the time $T$ takes at most $2\log k+\beta$ different values. Hence, for fixed $k$, we have
$$\#\hat{\mc{I}}_{k,p}\sqcap I_{p}\leq 4E_l\beta e^{2l\beta}\kappa c_2^{-1}\log\log L\frac{|I_p|}{|I_k|}.$$
It follows that for $k\geq 2A_{l}$, the $k$-th local characteristic of $\mc{K}(\mc{I}_{k})$ (see (\ref{eq:localchar})) satisfies
\begin{align}
 & \Delta_{k}\leq \left(\prod_{i=p}^{k-1}\frac{4}{r_{i}}\right)\max_{I_{p}\in\mc{I}_{p}}\#\hat{\mc{I}}_{k,p}\sqcap I_{p}\leq\frac{4^{A_{l}-1}}{r_{k-A_{l}+1}\dotsm r_{k-1}}\nonumber \\
 & \cdot 4E_l\beta e^{2l\beta}\kappa c_2^{-1}\log\log L\frac{L(r_{0}\dotsm r_{k-A_{l}})^{-1}}{L(r_{0}\dotsm r_{k-1})^{-1}}\nonumber \\
 & =4^{A_{l}}E_l\beta e^{2l\beta}\kappa c_2^{-1}\log\log L.\label{eq:dn1}
\end{align}
On the other hand, for $k<2A_{l}$ we have
\begin{align}
 & \Delta_{k}\leq \left(\prod_{i=0}^{k-1}\frac{4}{r_{i}}\right)\#\hat{\mc{I}}_{k,0}\sqcap I_{0}\leq \frac{4^{2A_{l}-1}}{r_{0}\dotsm r_{k-1}}\nonumber \\
 & \cdot E_l\beta e^{2l\beta}\kappa c_2^{-1}\log\log L\frac{L}{L(r_{0}\dotsm r_{k-1})^{-1}}\nonumber \\
 & =4^{2A_{l}}E_l\beta e^{2l\beta}\kappa c_2^{-1}\log\log L.\label{eq:dn2}
\end{align}
Let us pick $\varepsilon>0$. Then we can always choose the parameters $\beta$, $\kappa$, and $L$ so that\footnote{In the fifth condition we are using the fact that $R^{(*)}(1)\geq \beta$ when $\kappa\leq e^{-(l+1)\beta}$, as seen in Lemma \ref{lem:danlength}}
$$
\begin{cases}
4\max\{|\log\kappa|,l\log\beta\}\leq e^\beta & \mbox{from Lemmas \ref{lem:deriv} and \ref{lem:computations}}\\
\kappa\leq {\gamma} & \mbox{from Lemma \ref{lem:danlengthdual}}\\
\kappa\leq e^{-(l+1)\beta} & \mbox{from Lemma \ref{lem:danlength} and Corollary \ref{cor:RvsRstar}}\\
c_{2}\ll_{l} 1 & \mbox{from Lemmas \ref{lem:alblocksdual}, \ref{lem:alblockssim}, and \ref{lem:alblockssim2}}\\
c_{2}\geq e^{-(l+1)\beta}\geq e^{-(l+1)R^{(*)}(0)} & \mbox{from Lemmas \ref{lem:alblocksdual}, \ref{lem:alblockssim}, and \ref{lem:alblockssim2}}\\ 
c_{2}\geq 4e^{-\beta} & \mbox{from Lemma \ref{lem:alblockssim2}}\\
e^{100l\beta}\geq  4^{(A_l-2)/(C_l+1)}\kappa^{-1} & \mbox{from (\ref{eq:cond2})}\\
L\geq e^{2B_{l}\beta}\left(\kappa^{-1}(2A_{l})^{l-1}\log (2A_{l})^{l}\right)^{C_{l}+1} & \mbox{from (\ref{eq:longI0})}
\end{cases}
$$
and additionally
$$4^{2A_l}E_l\beta e^{2l\beta}\kappa c_2^{-1}\log\log L\leq \varepsilon.$$
To see this, one may start by fixing $c_2$ as a small constant depending only on $l$. Then it will be convenient {to} express both $\kappa$ and $L$ as a power of $e^{\beta}$, with $\beta$ large enough in terms of $l$ and $\gamma$, e.g., $\kappa=e^{-3l\beta}$ and $L=e^{\beta}$. This will ensure that the large system of inequalities holds. Finally it is easily seen that the last inequality reduces to
$$e^{2l\beta}P(\beta)\kappa\ll_{l}\varepsilon,$$
where $P$ is some polynomial of degree and coefficients depending only on $l$. If $\kappa\geq e^{-3l\beta}$, this is satisfied for $\beta$ sufficiently large in terms of $\varepsilon$. Thus, by (\ref{eq:dn1}) and (\ref{eq:dn2}), we have that $\Delta_{k}\leq\varepsilon$ for all $k\geq 2$. For $k\leq 2$ no intervals are removed, by (\ref{eq:Trange1}) and the fact that we may always assume $T\geq t-R^{(*)}(t)\geq 0$ (see Corollary \ref{cor:C3}). Hence, $\Delta_{k}=0$. This shows that the Cantor-type set $\mc{K}(\mc{I}_{k})$ that we have constructed is Cantor-rich (see Definition \ref{def:Cantorrich}), completing the proof of Proposition \ref{prop:complementCR}.

\subsection{Proof of Lemma \ref{lem:computations}}
We conclude this section by proving Lemma \ref{lem:computations}. By (\ref{eq:Trange1}) and (\ref{eq:logFn}), we have that
\begin{equation}
\label{eq:1}
(k-2)\beta-\log L\leq T\leq k(2\log k+\beta).    
\end{equation}
Moreover, from (\ref{eq:111}) we have
\begin{equation}
\label{eq:upper8.1}
e^{(l+1)R^{(*)}(t)}\leq\kappa^{-1}\max\{t,\beta\}^{l-1}\log\max\{t,\beta\}.
\end{equation}
Now, (\ref{eq:tvsRt}) implies that, when $t\geq 4\max\{|\log\kappa|,l\log\beta\}$,
$$t/2\leq t-R^{(*)}(t)\leq T.$$
Assuming that $4\max\big\{|\log\kappa|,l\log\beta\big\}\leq e^{\beta}$, we conclude that for $t\geq e^\beta$ it holds
\begin{multline*}
e^{(l+1)R^{(*)}(t)}\leq\kappa^{-1}\max\{t,\beta\}^{l-1}\log\max\{t,\beta\} \\
\leq 2^{l}\kappa^{-1}T^{l-1}\log^{+} T\ll_{l}\kappa^{-1}\beta^{l}(k\log k)^{l-1}\log k.   \end{multline*}
 For $t\leq e^{\beta}$, we deduce from (\ref{eq:upper8.1}) that
 $$e^{(l+1)R^{(*)}(t)}\leq \kappa^{-1}e^{l\beta}.$$
Both of these expressions are then bounded above by $\kappa^{-1}e^{l\beta}(k\log k)^{l-1}\log k$, whence the upper bound. For the lower bound, we observe that, by (\ref{eq:222}), for $t\geq e^\beta\geq 4\max\{|\log\kappa|,l\log\beta\}$ it holds
$$e^{(l+1)R^{(*)}(t)}\geq 2^{-l}\kappa^{-1}\max\{t,\beta\}^{l-1}\log\max\{t,\beta\}.$$
Moreover, since $t_1\leq t+(l-1)R^{(*)}(t)\leq lt$, we have that $T\ll_{l} t$, whence for $t\geq e^{\beta}$
$$e^{(l+1)R^{(*)}(t)}\gg_{l}\kappa^{-1}T^{l-1}\log^{+}T.$$
On the other hand, for $T\ll_{l} t\leq e^{\beta}$ one trivially has
$$e^{(l+1)R^{(*)}(t)}\gg_{l}\kappa^{-1}e^{-\beta l}T^{l-1}\log^{+}T,$$
which works as a lower bound in both cases. Finally, for $k\geq \log L +2$, by (\ref{eq:1}) we find $T\geq k$, whence $\log ^{+}T\geq\log n$ and 
$$e^{(l+1)R^{(*)}(t)}\gg_{l}\kappa^{-1}e^{-\beta l}T^{l-1}\log k.$$
It follows that for any value of $k$ we have
$$e^{(l+1)R^{(*)}(t)}\gg_{l}\kappa^{-1}e^{-\beta l}T^{l-1}\frac{\log k}{\log\log L},$$
concluding the proof.

\appendix

\section{Lattice-Point Counting}
\label{sec:AppA}

In this section we gather a few results in the geometry of numbers that 
{are used} in the proof of our main theorem. For any lattice $\Lambda\subset\mb{R}^{d}$ we will denote by $\lambda_{i}$  ($i=1,\dotsc,d$) the first minimum {(i.e., the length of any shortest non-zero vector)} of the lattice $\bigwedge^{i}\Lambda\subset \mb{R}^{\binom{d}{i}}$ with respect to the supremum norm.  
We will also denote by $\delta_{i}$ for $i=1,\dotsc,d$ the successive minima of the lattice $\Lambda$. These are the quantities
$$\delta_{i}=\min\big\{\delta>0:\textup{rk}\left(\Lambda\cap [-\delta,\delta]^d\right)\geq i\big\},$$
where $\textup{rk}$ stands for the 
{dimension of the $\R$-span}. Note that $\delta_{1}=\lambda_{1}$. The co-volume of the lattice $\Lambda$ will be indicated by $\det\Lambda$.

We will make repeated use of the following classical result by Minkowski. See also {\cite[Chapter VIII,  equations (12) and (13)]{Cas}}. 

\begin{theorem}[Minkowski's Second Theorem]
\label{thm:Min2}
Let $\Lambda\subset\mb{R}^{d}$ be a 
{lattice}. Then 
$$\delta_{1}\dotsm\delta_{d}\asymp_{d}\det\Lambda.$$
\end{theorem}

Another fundamental result will be the following theorem by Blichfeldt \cite{Bli}.

\begin{theorem}[Blichfeldt]
\label{thm:Bli}
Let $K\subset \mb{R}^{d}$ be a bounded convex body, and let $\Lambda\subset\mb{R}^{d}$ be a lattice such that ${\textup{rk}}(\Lambda\cap K)=d$. Then
$$\#(\Lambda\cap K)\leq d!\frac{\textup{Vol}(K)}{\det\Lambda}+d.$$
\end{theorem}

From these, we deduce the following corollary.

\begin{cor}
\label{cor:counting}
Let $\Lambda\subset\mb{R}^{d}$ be a lattice, and let $$B:= [-b_{1},b_{1}]\times\dotsb\times[-b_{d},b_{d}]\subset\mb{R}^{d}$$ with $b_{i}>0$ for $i=1,\dotsc,d$. Select a permutation $\sigma\in S_{d}$ such that $b_{\sigma (1)}\geq\dotsb\geq b_{\sigma (d)}$, and assume that ${\textup{rk}}(\Lambda\cap B)=r\leq d$. Then there exists a constant $C\geq 1$ (depending solely on $d$) such that
$$\#(\Lambda\cap B)\leq C\left(1+\frac{b_{\sigma (1)}\dotsm b_{\sigma (r)}}{\lambda_{r}}\right).$$
\end{cor}

The next theorem, proved by Mahler \cite{Mah}, relates the minima of a lattice and the minima of its dual. This result will be crucial in Section \ref{sec:countingonav}. See also \cite[Chap.\ VIII, Sect.\ 3, Thm.\ VI]{Cas}.

\begin{theorem}[Mahler]
\label{thm:Ban}
Let $\Lambda\subset\mb{R}^{d}$ be a 
{lattice} and let $\Lambda^{*}$ be its dual lattice, i.e., {$\Lambda^{*}:=(g^{-1})^{\scriptscriptstyle{T}}\Z^d$ where $g$ is such that $\Lambda = g\Z^d$}. 
Then for $i=1,\dotsc,d$ it holds
$$\delta_{i}\delta_{d+1-i}^{*}\asymp 1,$$
where $\delta_{i}^{*}$ denotes the $i$-th successive minimum of the lattice $\Lambda^{*}$.
\end{theorem}

\ignore{\section{Cantor-Type Sets}
\label{sec:AppB}

In the following section, we recall the definition of {a} Cantor-type set (introduced in \cite{BV}) and we gather some important results concerning these sets. The notation is borrowed from \cite{Ber15} (see Section 5). 

Let $r_{k}$ be a non-decreasing sequence of real numbers $\geq 1$. We call a sequence $\mc{I}_{k}$ of interval collections in $\mb{R}$ an $r_{k}$-sequence if $\mc{I}_{k+1}\subset r_{k}^{-1}\mc{I}_{k}$ for all $n=0,1,\dotsc$. We define
$$\hat{\mc{I}}_{k}:=\frac{1}{r_{k-1}}\mc{I}_{k-1}\setminus\mc{I}_{k}$$
and the Cantor-type set associated to $\mc{I}_{k}$ as
$$\mc{K}(\mc{I}_{k}):=\bigcap_{n\geq 0}\bigcup_{I\in\mc{I}_{k}}I.$$
Any set constructed through this procedure is called an $r_{k}$-Cantor-type set.

For an interval $J\subset\mb{R}$ and a collection of intervals $\mc{I}'$ in $\mb{R}$ we set
$$\mc{I}'\sqcap J:=\{I\in\mc{I}':I\subset J\}.$$
We define the $n$-th local characteristic of the family $\mc{I}_{k}$ as
\begin{equation}
\label{eq:localchar}
d_{n}:=\min_{\left\{\hat{\mc{I}}_{k,p}\right\}}\sum_{p=0}^{n-1}\left(\prod_{i=p}^{n-1}\frac{4}{r_{i}}\right)\max_{I_{p}\in\mc{I}_{p}}\#\hat{\mc{I}}_{k,p}\sqcap I_{p},
\end{equation}
where $\left\{\hat{\mc{I}}_{k,p}\right\}$ varies through the partitions of the collection $\hat{\mc{I}}_{k}$ into $n$ subsets ($p=0,\dotsc,n-1$). Finally we define the global characteristic of the sequence $\{\mc{I}_{k}\}$ as
$$d:=\sup_{n\geq 0}d_{n}.$$

\begin{definition}
\label{def:Cantorrich}
A set $A\subset\mb{R}$ is said to be $r_{k}$-Cantor-rich if for any $\varepsilon>0$ there exists an $r_{k}$-Cantor-type set $\mc{K}(\mc{I}_{k})\subset A$ such that $\mc{I}_{k}$ has global characteristic $d<\varepsilon$. 
\end{definition}

For Cantor-rich sets we have the following adaptation of \cite{Ber15} Theorem 7.

\begin{prop}
The intersection of any finite number of $r_{k}$-Cantor-rich sets with same initial interval collection $\mc{I}_{0}$ is $r_{k}$-Cantor-rich.
\end{prop}

The next proposition is Theorem 4 in \cite{BV} and will be crucial in the proof of Theorem \ref{prop:mainres}.

\begin{prop}[Badziahin-Velani]
\label{prop:BV}
Let $\mc{K}(\mc{I}_{k})\subset\mb{R}$ be an $r_{k}$-Cantor-type set. If the global characteristic $d$ of $\mc{I}_{k}$ is less or equal to $1$, then
$$\dim\mc{K}(\mc{I}_{k})\geq\liminf_{n\to \infty}1-\frac{\log 2}{\log r_{k}},$$
where $\dim$ denotes the Hausdorff dimension.
\end{prop}

The two results stated above imply the following corollary.
 
\begin{cor}
\label{cor:CRFullHaus}
Let $r_{k}$ be a sequence of positive numbers tending to $\infty$. Then the intersection of a finite number of $r_{k}$-Cantor-rich sets with same initial interval collection $\mc{I}_{0}$ has full Hausdorff dimension. 
\end{cor}}

\end{document}